\documentclass[12pt]{amsart}
\usepackage{amsfonts}
\usepackage{amsmath,amssymb,amsthm,latexsym,graphicx,graphics}
\usepackage{multicol}
\usepackage{color}
\usepackage{url,hyperref}
\usepackage{euscript}
\usepackage{wrapfig}
\usepackage{caption}
\usepackage{url,hyperref}
\usepackage{euscript,color}
\usepackage[pagewise]{lineno}

\newtheorem{theorem}{Theorem}

\newtheorem{lemma}[theorem]{Lemma}

\newtheorem{remark}[theorem]{Remark}

\newtheorem{proposition}[theorem]{Proposition}

\hypersetup{
    bookmarks=true,             unicode=false,              pdftoolbar=true,            pdfmenubar=true,            pdffitwindow=false,         pdfstartview={FitH},        pdftitle={My title},        pdfauthor={Author},         pdfsubject={Subject},       pdfcreator={Creator},       pdfproducer={Producer},     pdfkeywords={keyword1} {key2} {key3},     pdfnewwindow=true,          colorlinks=true,           linkcolor=black,              citecolor=green,            filecolor=magenta,          urlcolor=blue         }

\definecolor{red}{rgb}{1,0,0}
\hypersetup{
    bookmarks=false,             unicode=false,              pdftoolbar=true,            pdfmenubar=true,            pdffitwindow=false,         pdfstartview={FitH},        pdftitle={My title},        pdfauthor={Author},         pdfsubject={Subject},       pdfcreator={Creator},       pdfproducer={Producer},     pdfkeywords={keyword1} {key2} {key3},     pdfnewwindow=true,          colorlinks=true,           linkcolor=black,              citecolor=green,            filecolor=magenta,          urlcolor=red         }

\begin{document}
\title[Harmonic unit normal sections of Grassmannians]{Harmonic unit normal
sections of Grassmannians associated with cross products}
\author{Francisco Ferraris, Ruth Paola Moas and Marcos Salvai}
\thanks{This work was supported by Consejo Nacional de Investigaciones Cient%
\'{\i}ficas y T\'{e}cnicas and Secretar\'{\i}a de Ciencia y T\'{e}cnica de
la Universidad Nacional de C\'{o}rdoba.}

\begin{abstract}
Let $G\left( k,n\right) $ be the Grassmannian of oriented subspaces of
dimension $k$ of $\mathbb{R}^{n}$ with its canonical Riemannian metric. We
study the energy of maps assigning to each $P\in G\left( k,n\right) $ a unit
vector normal to $P$. They are sections of a sphere bundle $E_{k,n}^{1}$
over $G\left( k,n\right) $. The octonionic double and triple cross products
induce in a natural way such sections for $k=2$, $n=7$ and $k=3$, $n=8$,
respectively. We prove that they are harmonic maps into $E_{k,n}^{1}$
endowed with the Sasaki metric. This, together with the well-known result
that Hopf vector fields on odd dimensional spheres are harmonic maps into
their unit tangent bundles, allows us to conclude that all unit normal
sections of the Grassmannians associated with cross products are harmonic.
In a second instance we analyze the energy of maps assigning an orthogonal
complex structure $J\left( P\right) $ on $P^{\bot }$ to each $P\in G\left(
2,8\right) $. We prove that the one induced by the octonionic triple product
is a harmonic map into a suitable sphere bundle over $G\left( 2,8\right) $.
This generalizes the harmonicity of the canonical almost complex structure
of $S^{6}$.
\end{abstract}

\maketitle


\medskip


\section{Introduction and presentation of the results}

One can say that Herman Gluck and Wolfgang Ziller set in trend the problem
of finding the best organized among all geometric structures of a certain
type on a manifold, when they proved that Hopf vector fields on $S^{3}$ have
minimum volume among all unit vector fields on $S^{3}$ \cite{GZ}. Shortly
after that, Eugenio Calabi and Herman Gluck found that the almost complex
structure on $S^{6}$ induced by the octonionic cross product is the best
orthogonal almost-complex structure on $S^{6}$, also with respect to volume 
\cite{CG}.\ Later, other criteria for the niceness of a geometric structure
on a manifold $M$ (thought of as a section of a fiber bundle over $M$) were
presented and also extensively studied.\ For instance, minimal or critical
total bending (a functional which measures to which extent the section fails
to be parallel), beginning with Gerrit Wiegmink \cite{wiegmink}, and minimal
or critical energy. Contributions in this sense were made for a wide variety
of structures (mainly for unit vector fields, but also, for example, for
distributions, metric contact structures, etc.), see for instance \cite{BGV,
PerroneLibro, GMO, olga, GMO1, GD1, Lou}.

Next we comment on the scope of the article. The general question, stated
vaguely, is the following:\ 

\begin{center}
What is the best way of assigning a unit vector $u\in P^{\bot }$ \\[0pt]
to each oriented subspace $P$ of dimension $k$ in $\mathbb{R}^{n}$?
\end{center}

The same question in a different guise:\ In which unit normal direction $%
u\left( P\right) $ should be moved each oriented $k$-subspace $P$ of $%
\mathbb{R}^{n}$ in order to obtain the best disposition of oriented affine $%
k $-subspaces at unit distance from the origin?

Within this setting we can ask another general question:

\begin{center}
What is the best way of assigning

to each oriented subspace $P$ of dimension $k$ in $\mathbb{R}^{k+2m}$

an orthogonal transformation $J\left( P\right) :P^{\bot }\rightarrow P^{\bot
}$ with $J\left( P\right) ^{2}=-  \operatorname{id}$?
\end{center}

These problems are stated more precisely considering sections of sphere
bundles over Grassmannians and choosing the criterion for nice organization,
normally minimal or critical energy or volume.

Sometimes there are no assignments like those in the first question if they
are required to be continuous. For instance, oriented one dimensional
subspaces of $\mathbb{R}^{3}$ may be identified with points of the
two-sphere, and their orthogonal planes with the corresponding tangent
spaces, but $S^{2}$ admits no continuous unit vector field. For the second
question, there may also be no continuous assignment like that, for example
for the case $k=1$ and $2m\neq 2,6$. In fact, the existence would contradict
Borel and Serre's result in \cite{BS}, that the sphere $S^{n}$ admits an
almost complex structure if and only if $n=2$ or $n=6$.

We are far from answering the questions in general. We find assignments that
are geometrically distinguished rather than optimal. Our contribution to the
problem in the first question is as follows: We concentrate on the cases
where the mappings are given in terms of cross products and prove that they
are harmonic from the Grassmannian into a conveniently defined sphere bundle
over it.

One of our two main theorems generalizes amply the classical result
asserting that Hopf vector fields on odd dimensional spheres are harmonic
maps into the unit tangent bundle endowed with the Sasaki metric \cite{Han}
(see also \cite{Perrone, PerroneCalvaruso}). Such a unit vector field is a
critical point of the energy functional if one considers variations through 
\emph{all} smooth mappings. We remark that this condition is stronger than
another one, which has been also studied and sometimes is called \emph{%
vertical harmonicity}, where only variations through unit vector fields are
considered. The vertical harmonic maps turn out to be critical for the total
bending functional.

Now we deal with the second question. Let $G\left( 2,8\right) $ be the
Grassmannian of oriented subspaces of $\mathbb{R}^{8}$ of dimension $2$. We
consider mappings assigning to each $P\in G\left( 2,8\right) $ a
skew-symmetric transformation $T$ on $P^{\bot }$ with $ \operatorname{tr}\left(
T^{2}\right) =-6$. We prove that the distinguished map of this type induced
by the triple octonionic cross product (here $T$ is in particular a complex
orthogonal structure) is harmonic into a certain spherical bundle over $%
G\left( 2,8\right) $. The similar problem with the double octonionic cross
product yields the harmonicity of the almost complex canonical structure of
the sphere $S^{6}$, proved in \cite{CP}, and is related to the main result
in \cite{BLS}, that this structure has minimal energy (but varying only
through almost complex structures).

In the rest of the introduction we recall the relevant definitions and state
Theorems \ref{general} and \ref{Jarmo}, our contributions to the first and
second question above, respectively. The corresponding details are given in
Section 2. In Sections 3 and 4 the theorems are proved. Throughout the
article, smooth means of class $C^{\infty }$.

\subsection{The total bending and the energy of unit sections}

Let $M$ be an oriented compact Riemannian manifold. Let $E\rightarrow M$ be
a Riemannian vector bundle over $M$ with a metric connection $\nabla $ and
let $E^{1}=\left\{ v\in E\mid \left\Vert v\right\Vert =1\right\} $ be the
corresponding spherical bundle.

The functional $\mathcal{B}$, called \emph{total bending}, assigns to each
section $\sigma $ of $E^{1}\rightarrow M$ the number 
\begin{equation}
\mathcal{B}\left( \sigma \right) =\int_{M}\left\Vert \nabla \sigma
\right\Vert ^{2}\,dv\text{,}  \label{combadura}
\end{equation}%
that indicates to which extent the section $\sigma $ fails to be parallel.
Here $dv$ is the volume form of $M$ and $\left\Vert T\right\Vert ^{2}= \operatorname{%
tr}\left( T^{t}T\right) $ ($T^{t}$ denotes the transpose of $T$).

The \emph{energy} of a smooth map $F:M\rightarrow N$, with $N$ a Riemannian
manifold, is by definition the integral 
\begin{equation*}
\mathcal{E}\left( F\right) =\dfrac{1}{2}\int_{M}\left\Vert dF\right\Vert
^{2}\,dv\text{.}
\end{equation*}%
Critical points of the energy functional are called \emph{harmonic maps}.

Now we endow $E^{1}$ with the Sasaki metric induced by $\nabla $. Then it
makes sense to consider the energy of the sections $\sigma :M\rightarrow
E^{1}$. It turns out that $\mathcal{E}\left( \sigma \right) $ differs from $%
\mathcal{B}\left( \sigma \right) $ in a pair of constants involving only the
dimension and the volume of $M$.

If a section is a harmonic map, then it is critical for $\mathcal{B}$, but
we want to emphasize that the concept of critical section for the energy is
in general much stronger than the corresponding notion for the total
bending, since in the first case arbitrary smooth functions from the base
into the spherical bundle are admitted as variations, not only variations by
sections. For this reason, critical points of $\mathcal{B}$ are called \emph{%
vertically harmonic maps}.

\subsection{Unit normal sections of the Grassmannian}

Let $G\left( k,n\right) $ be the Grassmannian of oriented $k$-dimensional
subspaces of $\mathbb{R}^{n}$ endowed with the canonical symmetric
Riemannian metric. Let%
\begin{equation}
E_{k,n}=\left\{ (P,v)\in G\left( k,n\right) \times \mathbb{R}^{n}\mid v\text{
is orthogonal to }P\right\} \text{,}  \label{bundle}
\end{equation}%
which is the total space of a Riemannian vector bundle over $G\left(
k,n\right) $ with typical fiber $\mathbb{R}^{n-k}$ (the fiber inherits the
inner product from $\mathbb{R}^{n}$).

The Riemannian vector bundle $E_{k,n}\rightarrow G\left( k,n\right) $ has a
canonical \emph{metric connection}, whose associated covariant derivative is
the following: If $Q:I\rightarrow G\left( k,n\right) $ is a smooth curve of
oriented subspaces in\ $\mathbb{R}^{n}$ and $x:I\rightarrow \mathbb{R}^{n}$
is a smooth curve such that $x_{t}\perp Q_{t}$ for all $t$, then%
\begin{equation}
\frac{D}{dt}\left( Q_{t},x_{t}\right) =\left( Q_{t},\pi _{t}\left(
x_{t}^{\prime }\right) \right) \text{,}  \label{DerivadaCovariante}
\end{equation}%
where $x^{\prime }$ is the usual derivative of $x$ in $\mathbb{R}^{n}$ and $%
\pi _{t}$ is the orthogonal projection onto $\left( Q_{t}\right) ^{\bot }$.

We are looking for unit normal sections of the Grassmannian $G\left(
k,n\right) $, that is, sections of the sphere bundle 
\begin{equation}
E_{k,n}^{1}=\left\{ \left( P,v\right) \in E_{k,n}\mid \left\Vert
v\right\Vert =1\right\} \rightarrow G\left( k,n\right) \text{,}
\label{bundle1}
\end{equation}%
that have critical energy.

\subsection{Unit normal sections associated with cross products\label%
{crossproduct}}

Cross products provide distinguished examples of sections of $%
E_{k,n}^{1}\rightarrow G\left( k,n\right) $ for some $\left( k,n\right) $.
Robert\ Brown and Alfred\ Gray defined and classified them in \cite{brown}
(see also \cite{gray}): Let $V$ be an $n$-dimensional vector space over $%
\mathbb{R}$ and let $\left\langle \cdot ,\cdot \right\rangle $ be a positive
definite inner product on $V$. A \emph{cross product }on\emph{\ }$V$ is a
multilinear map $X:V^{r}\rightarrow V\;\;$($1\leq r\leq n$) satisfying%
\begin{equation*}
\left\langle X\left( u_{1},\dots ,u_{r}\right) ,u_{i}\right\rangle =0\text{\
\ \ \ \ \ and\ \ \ \ \ \ }\left\Vert X\left( u_{1},\dots ,u_{r}\right)
\right\Vert ^{2}={\mathrm{det}}\left( \left\langle u_{i},u_{j}\right\rangle
\right)
\end{equation*}%
for any $r$-tuple $u_{1},\dots ,u_{r}$ in $V$.

Cross products exist only when $\left( r,n\right) $ equals $\left(
m,m+1\right) $, $\left( 1,2m\right) $ (any $m\in \mathbb{N}$), $\left(
2,7\right) $ or $\left( 3,8\right) $. The first case is trivial and in the
second case, $X$ is an orthogonal transformation satisfying $X^{2}=-  \operatorname{id%
}$. By the classification of Brown and Gray there remain (up to
(anti-)isomorphisms) only the canonical cross products $X_{2,7}$ and $%
X_{3,8} $ on the imaginary octonions $\mathbb{R}^{7}=  \operatorname{Im}\mathbb{O}$
and the octonions $\mathbb{O}$, respectively, given by%
\begin{align}
X_{2,7}\left( u,v\right) & =u\times v=uv+\left\langle u,v\right\rangle \text{%
,}  \label{X2,7} \\
X_{3,8}\left( u,v,w\right) & =-u\left( \overline{v}w\right) +\left\langle
u,v\right\rangle w+\left\langle v,w\right\rangle u-\left\langle
w,u\right\rangle v  \label{X3,8}
\end{align}%
($uv$ denotes the multiplication in $\mathbb{O}$). They are called the
double and the triple cross products, respectively.

A cross product $X:\left( \mathbb{R}^{n}\right) ^{k}\rightarrow \mathbb{R}%
^{n}$ on $\mathbb{R}^{n}$ induces naturally the section%
\begin{equation*}
\sigma _{X}\left( Q\right) =\left( Q,X\left( u_{1},\dots ,u_{k}\right)
\right)
\end{equation*}%
of the sphere bundle $E_{k,n}^{1}\rightarrow G\left( k,n\right) $ as in (\ref%
{bundle1}), where $\left\{ u_{1},...,u_{k}\right\} $ is any positively
oriented orthonormal basis of $Q$. The properties of cross products imply
that $\sigma _{X}$ is well defined.

Now we are in the position to state one of our main results.

\begin{theorem}
\label{general}The sections of the sphere bundles $E_{k,n}^{1}\rightarrow
G\left( k,n\right) $ associated with cross products are harmonic maps.
\end{theorem}

\begin{remark}
\label{Hopf}The theorem generalizes the classical result that Hopf vector
fields on odd dimensional spheres are harmonic maps \emph{(}see \cite{Han}
and also \cite{Perrone, PerroneCalvaruso}\emph{)}. In fact, in the case $%
\left( 1,2m\right) $, the cross product $X$ is an orthogonal linear complex
structure on $\mathbb{R}^{2m}$ and it may be thought of as a Hopf vector
field on $S^{2m-1}$: $X\left( p\right) \in p^{\bot }=T_{p}S^{2m-1}$,
identifying the Grassmannian $G\left( 1,2m\right) $ with $S^{2m-1}$ and $%
E_{1,2m}^{1}$ with $TS^{2m-1}$. Previously, Wiegmink had proved in \cite%
{wiegmink} that Hopf fields are critical for the total bending functional 
\emph{(}see also \cite{wood}\emph{)}.
\end{remark}

By the remark, we concentrate only on the remaining nontrivial cases:\ We
prove below in Theorems \ref{TeoSigma2} and \ref{TeoSigma3} that the
following sections are harmonic maps:\ $\sigma _{2}:G\left( 2,7\right)
\rightarrow E_{2,7}^{1}$ and $\sigma _{3}:G\left( 3,8\right) \rightarrow
E_{3,8}^{1}$, defined by 
\begin{eqnarray*}
\sigma _{2}\left( u\wedge v\right) &=&\left( u\wedge v,u\times v\right) 
\text{,} \\
\sigma _{3}\left( u\wedge v\wedge w\right) &=&\left( u\wedge v\wedge
w,X_{3,8}\left( u,v,w\right) \right)
\end{eqnarray*}%
for orthonormal subsets $\left\{ u,v\right\} $ of $\mathbb{R}^{7}$ and $%
\left\{ u,v,w\right\} $ of $\mathbb{R}^{8}$, respectively.

\medskip

If there existed a parallel section of $E_{k,n}^{1}\rightarrow G\left(
k,n\right) $, then the energy would trivially attain a minimum at it, but
this is not the case:

\begin{proposition}
\label{ParallelSections} The spherical bundles $E_{2,7}^{1}\rightarrow
G\left( 2,7\right) $ and $E_{3,8}^{1}\rightarrow G\left( 3,8\right) $ do not
admit parallel sections, not even local ones.
\end{proposition}

\subsection{Orthogonal complex normal sections of the Grassmannian}

We call $ \operatorname{Skew}\left( \mathbb{R}^{8}\right) $ the vector space of all
the skew-symmetric linear operators on $\mathbb{R}^{8}$, endowed with the
inner product whose norm is defined by $\left\Vert T\right\Vert ^{2}=\tfrac{1%
}{6} \operatorname{tr}\left( T^{t}T\right) $. Given $P\in G\left( 2,8\right) $, we set%
\begin{equation}
 \operatorname{Skew}_{P}\left( \mathbb{R}^{8}\right) =\left\{ T\in  \operatorname{Skew}\left( 
\mathbb{R}^{8}\right) \mid \left. T\right\vert _{P}=0\right\} \text{.}
\label{SkewP}
\end{equation}%
Note that if $T\in  \operatorname{Skew}_{P}\left( \mathbb{R}^{8}\right) $, then $%
T\left( P^{\bot }\right) \subset P^{\bot }$. Then, $ \operatorname{Skew}_{P}\left( 
\mathbb{R}^{8}\right) $ can be canonically identified with the set of all
the skew-symmetric operators defined on $P^{\bot }$.

The canonical projection%
\begin{equation}
E=_{\text{def}}\left\{ \left( P,T\right) \mid P\in G\left( 2,8\right) \;%
\text{ and }\;T\in  \operatorname{Skew}_{P}\left( \mathbb{R}^{8}\right) \right\}
\rightarrow G\left( 2,8\right)  \label{E2,8}
\end{equation}%
admits a Riemannian vector bundle structure (the fiber inherits the inner
product from $ \operatorname{Skew}\left( \mathbb{R}^{8}\right) $). On this bundle we
have a canonical metric connection, whose associated covariant derivative is
the following: If $t\mapsto P_{t}\in G\left( 2,8\right) $ is a smooth curve
of oriented planes in $\mathbb{R}^{8}$ and $t\mapsto T_{t}$ is a smooth
curve in $ \operatorname{Skew}\left( \mathbb{R}^{8}\right) $ with $T_{t}\ $vanishing
on $P_{t}$, then 
\begin{equation}
\dfrac{D}{dt}\left( P_{t},T_{t}\right) =\left( P_{t},\Pi _{P_{t}}\left( 
\frac{d}{dt}T_{t}\right) \right) \text{,}  \label{DerCovJ}
\end{equation}%
where $\Pi _{P}$ is the orthogonal projection of $ \operatorname{Skew}\left( \mathbb{R%
}^{8}\right) $ onto $ \operatorname{Skew}_{P}\left( \mathbb{R}^{8}\right) $.

We denote by $E^{1}=\left\{ \left( P,T\right) \in E\mid \left\Vert
T\right\Vert =1\right\} \rightarrow G\left( 2,8\right) $, the spherical
subbundle of $E\rightarrow G\left( 2,8\right) $.

From \cite{Fei} we learned of a distinguished section of $E^{1}$, associated
with the triple cross product, defined by%
\begin{equation}
\mathfrak{J}:G\left( 2,8\right) \rightarrow E^{1}\text{,\ \ \ \ \ \ \ \ \ \
\ }\mathfrak{J}\left( u\wedge v\right) =\left( u\wedge v,J_{u\wedge
v}\right) \text{,}  \label{Jota}
\end{equation}%
where $J_{u\wedge v}\in  \operatorname{Skew}_{u\wedge v}\left( \mathbb{R}^{8}\right) $
is given by $J_{u\wedge v}\left( w\right) =X_{3,8}\left( u,v,w\right) $ and
may be thought of as an orthogonal linear complex transformation on $\left(
u\wedge v\right) ^{\bot }$.

Our main result in relation to $\mathfrak{J}$ is as follows.

\begin{theorem}
\label{Jarmo}The section $\mathfrak{J}$ is a harmonic map.
\end{theorem}

For the other cross products, orthogonal complex normal sections of
Grassmannians can be constructed in an analogous manner. Apart from empty or
trivial cases, only one remains, which corresponds, via an appropriate
identification, with the canonical almost complex structure of the sphere $%
S^{6}$, whose harmonicity is known from \cite{CP} (see also \cite{BLS}).
More details are given in Subsection \ref{final}.

\section{Preliminaries\label{preliminares}}

\subsection{Critical sections for the total bending and the energy}

We recall from Proposition 3.2 in \cite{paco} necessary and
sufficient conditions for a section of a sphere bundle to be critical for
the total bending or the energy functionals. Let $\pi :E\rightarrow M$ be a
Riemannian vector bundle with a metric connection over a compact oriented
Riemannian manifold. The \emph{rough Laplacian} is the operator $\Delta $
acting on smooth sections of $E$ defined by%
\begin{equation}
\Delta \sigma =  \operatorname{tr}\left( \nabla ^{2}\sigma \right) \text{,}
\label{Burdo}
\end{equation}%
where $\nabla _{X,Y}^{2}\sigma =\nabla _{X}\nabla _{Y}\sigma -\nabla
_{\nabla _{X}Y}\sigma $ for local vector fields $X,Y$ on $M$. We consider
the curvature tensor to be defined by $R_{X,Y}\sigma =\nabla _{\left[ X,Y%
\right] }+\left[ \nabla _{Y},\nabla _{X}\right] $.

\begin{proposition}
\label{PacoCarmelo} \emph{\cite{paco}} Let $\Pi :E\rightarrow M$ be a
Riemannian vector bundle with a metric connection $\nabla $ over an oriented
compact Riemannian manifold and let $\sigma :M\rightarrow E^{1}$ be a
section of the corresponding sphere bundle.

\emph{a)} The section $\sigma $ is vertically harmonic if and only if there
is a smooth real function $f$ on $M$ such that%
\begin{equation}
\Delta \sigma =f\sigma \text{.}  \label{NablaSigma}
\end{equation}

\emph{b)} Suppose additionally that $E$ is endowed with the Sasaki metric.
The section $\sigma $ is a harmonic map if and only if $\sigma $ is
vertically harmonic and also $\mathcal{R}_{\sigma }\equiv 0$, where $%
\mathcal{R}_{\sigma }$ is the $1$-form on $M$ defined by%
\begin{equation}
\mathcal{R}_{\sigma }\left( X\right) =\sum\nolimits_{i=1}^{m}\left\langle
R_{X,e_{i}}\sigma ,\nabla _{e_{i}}\sigma \right\rangle \text{,}
\label{Rsigma}
\end{equation}%
for any vector field $X$\thinspace over $M$, with $\left\{ e_{1},\dots
,e_{m}\right\} $ a local orthonormal frame.
\end{proposition}

The condition in (a) was proved for the particular case where $E$ is a
tangent bundle by Wiegmink \cite{wiegmink} and Wood \cite{wood} for compact
manifolds and in general by Gil-Medrano \cite{olga}. One can find in \cite%
{Grigo} a recent generalization for an octonionic bundle over a $G_{2}$%
-manifold.

\subsection{Cross products\label{PCruz}}

We review the definition of the octonions and properties of the cross
products, based mainly on Chapter 6 of \cite{harvey}. The octonions $\mathbb{%
O}$ are a normed division algebra over the real numbers which is neither
commutative nor associative. Explicitly (see 17.2 of \cite{Sabinin}), the
octonions are Euclidean space $\mathbb{R}^{8}$ with its canonical inner
product and with a multiplication $\mathbb{R}^{8}\times \mathbb{R}%
^{8}\rightarrow \mathbb{R}^{8}$ extending bilinearly the one defined in the
canonical basis $\left\{ e_{i}\mid i=0,\dots ,7\right\} $ by $%
e_{0}e_{i}=e_{i}e_{0}=e_{i}$ for $i=0,...,7$ (in particular, $e_{0}$ is the
unity of the product), and also%
\begin{equation}
e_{i}e_{j}=-\delta _{ij}e_{0}+\varepsilon _{ijk}e_{k}\text{,}  \label{tabla1}
\end{equation}%
for $i,j,k\geq 1$, where $\varepsilon _{ijk}$ is a completely skew symmetric
tensor taking the value $+1$ when 
\begin{equation*}
ijk=123,145,176,246,257,347,365\text{.}
\end{equation*}

We recall some basic properties of the octonions that will be used
frequently in the following sections. If $\left\{ u,v\right\} $ is an
orthogonal subset of $ \operatorname{Im}\mathbb{O}$, then $\left\{ u,v,uv\right\} $
is orthogonal as well and the identities $uv=-vu$ and $u\left( uv\right) =-v$
hold. Also, for all orthogonal $u,v,w\in \mathbb{O}$ we have that $-\left( 
\bar{w}v\right) \bar{u}=\left( \bar{u}v\right) \bar{w}$ (it can be deduced
from Corollary 6.13 in \cite{harvey}).

Next we mention a property of the triple cross product (\ref{X3,8}), see for
instance (5.22) in \cite{salamon}: For any orthonormal set $\left\{
z,u,v,w\right\} $ of $\mathbb{R}^{8}$ we have that%
\begin{equation}
\left\langle X_{3,8}\left( u,v,w\right) ,z\right\rangle =-\left\langle
X_{3,8}\left( z,v,w\right) ,u\right\rangle \text{.}  \label{propiedadX}
\end{equation}%
Also, 
\begin{equation}
e_{i}\times e_{j}=X_{3,8}\left( e_{0},e_{i},e_{j}\right)  \label{DobleTriple}
\end{equation}%
holds for all $i,j=0,\dots ,7$. In particular, by equation (\ref{tabla1})
and the definition of the double cross product, we have $X_{3,8}\left(
e_{0},e_{1},e_{2}\right) =e_{1}\times e_{2}=e_{3}$.

\subsection{Grassmannians of oriented subspaces}

Let $G\left( k,n\right) $ be as before the set of all oriented vector
subspaces of dimension $k$ in $\mathbb{R}^{n}$. Given an orthonormal subset $%
\left\{ u_{1},\dots ,u_{k}\right\} \subset \mathbb{R}^{n}$, we define the
oriented subspace%
\begin{equation*}
u_{1}\wedge \dots \wedge u_{k}=\left(  \operatorname{span}\left\{ u_{1},\dots
,u_{k}\right\} ,u^{1}\wedge \dots \wedge u^{k}\right) \text{,}
\end{equation*}%
where $\left\{ u^{1},\dots ,u^{k}\right\} $ is the dual basis of $\left\{
u_{1},\dots ,u_{k}\right\} $. Then $G\left( k,n\right) $ consists of all $%
u_{1}\wedge \dots \wedge u_{k}$, with orthonormal $u_{1},\dots ,u_{k}\in 
\mathbb{R}^{n}$.

We consider on $G\left( k,n\right) $ the normal Riemannian metric and
present it as usual as the homogeneous symmetric space $G/H$, where $G=SO(n)$
and $H=SO(k)\times SO(n-k)$ is the isotropy subgroup at $e_{0}\wedge
...\wedge e_{k-1}$. Let $\mathfrak{g}$ and $\mathfrak{h}$ be the Lie
algebras of $G$ and $H$, respectively, and let $\mathfrak{g}=\mathfrak{h}%
\oplus \mathfrak{m}$ be the corresponding Cartan decomposition. We have the
canonical identification of $\mathfrak{m}$ with $T_{e_{0}\wedge ...\wedge
e_{k-1}}G\left( k,n\right) $ and it is well-known that the geodesics of $%
G\left( k,n\right) $ through $e_{0}\wedge ...\wedge e_{k-1}$ are exactly the
curves $t\mapsto \exp \left( tZ\right) e_{0}\wedge ...\wedge e_{k-1}$ for $%
Z\in \mathfrak{m}$. In particular, for $0\leq \ell <k$, $k\leq j<n$, the
curves $\gamma _{j}^{\ell }:\mathbb{R}\rightarrow G\left( k,n\right) $ given
by 
\begin{equation}
\gamma _{j}^{\ell }(t)=e_{0}\wedge \dots \wedge \left( \cos t\text{ }e_{\ell
}+\sin t\text{ }e_{j}\right) \wedge \dots \wedge e_{k-1}  \label{geodegrass}
\end{equation}%
($\cos t$ $e_{\ell }+\sin t$ $e_{j}$ occupies the $\ell $-th place in the
exterior product) are geodesics of $G\left( k,n\right) $. Their initial
velocities are%
\begin{equation}
e_{j}^{\ell }=e_{j}\otimes e^{\ell }-e_{\ell }\otimes e^{j}\in \mathfrak{m}%
\subset \mathfrak{o}\left( n\right)  \label{eELEj}
\end{equation}%
and form an orthonormal basis of $T_{e_{0}\wedge ...\wedge e_{k-1}}G\left(
k,n\right) $.

\section{Harmonic unit normal sections of the Grassmannians associated with
cross products}

Theorem \ref{general} above states that the sections of the spherical bundle 
$E_{k,n}^{1}\rightarrow G\left( k,n\right) $ associated with cross products
are harmonic maps. By Subsection \ref{crossproduct} we must only study the
cases $\left( m,m+1\right) $, $\left( 1,2m\right) $, $\left( 2,7\right) $
and $\left( 3,8\right) $. The first one is trivial and we addressed the
second one in Remark \ref{Hopf}. The remaining cases, namely $\left(
3,8\right) $ and $\left( 2,7\right) $, will be dealt with in two subsections
below.

\medskip

On the vector bundle $\Pi :E_{k,n}\rightarrow G\left( k,n\right) $ given in (%
\ref{bundle}) we have a canonical\emph{\ metric connection }$\nabla $,
defined as follows. If $Y$ is a vector field on $G\left( k,n\right) $ and $%
\sigma \in \Gamma \left( G\left( k,n\right) ,E_{k,n}\right) $, then%
\begin{equation*}
\left( \nabla _{Y}\sigma \right) _{P}=\left( P,\pi _{P}\left( \left(
dx_{\sigma }\right) _{P}\left( Y_{P}\right) \right) \right) \text{,}
\end{equation*}%
where $\sigma \left( Q\right) =\left( Q,x_{\sigma }\left( Q\right) \right)
\in \left( E_{k,n}\right) _{Q}$, with $x_{\sigma }:G\left( k,n\right)
\rightarrow \mathbb{R}^{n}$, and $\pi _{P}$ is the orthogonal projection of $%
\mathbb{R}^{n}$ onto $P^{\bot }$. The associated\emph{\ covariant} \emph{%
derivative} turns out to be the one in (\ref{DerivadaCovariante}).

\smallskip

Next we present a lemma that will be useful later.

\begin{lemma}
\label{derivadask,n}Let $P:\mathbb{R}^{2}\rightarrow G\left( k,n\right) $ be
a parametrized surface and let $\sigma =\left(  \operatorname{id},x\right) $ be a
smooth section of $\Pi :E_{k,n}^{1}\rightarrow G\left( k,n\right) $, with $%
x:G\left( k,n\right) \rightarrow \mathbb{R}^{n}$. We denote%
\begin{equation*}
S_{1}=\left. \frac{d}{ds}\right\vert _{0}x\left( P_{0,s}\right) \text{ }\ \
\ \ \text{and }\ \ \ \ S_{2}=\left. \frac{\partial ^{2}}{\partial t\partial s%
}\right\vert _{\left( 0,0\right) }x\left( P_{t,s}\right) \text{,}
\end{equation*}%
and let $\pi _{t}$ be the orthogonal projection of $\mathbb{R}^{n}$ onto $%
\left( P_{t,0}\right) ^{\bot }$. Then%
\begin{equation}
\left. \frac{D}{ds}\right\vert _{0}\sigma \left( P_{t,s}\right) =\left(
P_{t,0},\pi _{t}\circ \left. \frac{d}{ds}\right\vert _{0}x_{P_{t,s}}\right) 
\text{,}  \label{DS1}
\end{equation}%
\begin{equation}
\left. \frac{D^{2}}{dtds}\right\vert _{\left( 0,0\right) }\sigma \left(
P_{t,s}\right) =\left( P_{0,0},\pi _{0}\circ \left( \pi _{0}^{\prime }\circ
S_{1}+S_{2}\right) \right) \text{,}  \label{DS2}
\end{equation}%
where $\pi _{0}^{\prime }$ denotes the derivative at $t=0$ of the function $%
t\mapsto \pi _{t}\in \mathbb{R}^{n\times n}$.
\end{lemma}

\begin{proof}
The first assertion is an immediate consequence of (\ref{DerivadaCovariante}%
) and (\ref{DS2}) follows from%
\begin{equation*}
\pi _{0}\left( \left. \tfrac{d}{dt}\right\vert _{0}\left( \pi _{t}\circ
\left. \tfrac{d}{ds}\right\vert _{0}x\left( P_{t,s}\right) \right) \right)
=\pi _{0}\circ \left( \pi _{0}^{\prime }\circ S_{1}+\pi _{0}\circ
S_{2}\right) \text{.}\qedhere
\end{equation*}
\end{proof}

\subsection{Harmonicity of $\protect\sigma _{3}$}

We recall from (\ref{geodegrass}) the geodesics $\gamma _{j}^{\ell }$ of $%
G\left( 3,8\right) $ and their initial velocities $e_{j}^{\ell }$ ($\ell
=0,1,2$ and $j=3,\dots ,7$). We denote by $E_{j}^{\ell }$ the vector field
on a normal neighborhood of $e_{0}\wedge e_{1}\wedge e_{2}$ in $G\left(
3,8\right) $, such that $E_{j}^{\ell }\left( e_{0}\wedge e_{1}\wedge
e_{2}\right) =e_{j}^{\ell }$ and is parallel along the radial geodesics
starting from $e_{0}\wedge e_{1}\wedge e_{2}$. Since the Levi Civita
connection of $G\left( 3,8\right) $ is torsion free, for all $k,\ell ,i,j$
we have%
\begin{equation}
\left[ E_{i}^{k},E_{j}^{\ell }\right] \left( e_{0}\wedge e_{1}\wedge
e_{2}\right) =0\text{.}  \label{corIgualCero}
\end{equation}

From now on in this section, to simplify the notation, we sometimes write $X$
instead of $X_{3,8}$ and omit the foot points of elements of $E_{k,n}^{1}$.

\begin{lemma}
\label{dercov2}Let $\sigma $ be a section of $E_{3,8}^{1}$. For $k,\ell
=0,1,2$ and $i,j=3,\dots ,7$ we have%
\begin{equation*}
\nabla _{e_{i}^{k}}\nabla _{E_{j}^{\ell }}\sigma _{3}=\left. \frac{D^{\gamma
_{i}^{k}}}{dt}\right\vert _{0}\left( \left. \frac{D}{ds}\right\vert
_{0}\sigma _{3}\left( P_{t,s}\right) \right) \text{,}
\end{equation*}%
where $P_{t,s}=\exp \left( te_{i}^{k}\right) \,\gamma _{j}^{\ell }\left(
s\right) $.
\end{lemma}

\begin{proof}
The expression follows from the facts that $\gamma _{j}^{\ell }$ is the
geodesic of $G\left( 3,8\right) $ beginning at $e_{0}\wedge e_{1}\wedge
e_{2} $ with initial velocity $e_{j}^{\ell }$ and that $s\mapsto P_{t,s}$
takes the value $\gamma _{i}^{k}\left( t\right) $ at $s=0$ with initial
velocity $E_{j}^{\ell }\left( \gamma _{i}^{k}\left( t\right) \right) $.
Indeed, since $G\left( 3,8\right) $ is a symmetric space and $e_{j}^{\ell
}\in \mathfrak{m}$, the parallel transport along the curve $t\mapsto \gamma
_{i}^{k}\left( t\right) =\exp \left( te_{i}^{k}\right) \left( e_{0}\wedge
e_{1}\wedge e_{2}\right) $ between $0$ and $t$ is performed by $d\exp \left(
te_{i}^{k}\right) _{e_{0}\wedge e_{1}\wedge e_{2}}$. Therefore, as desired,%
\begin{equation*}
E_{j}^{\ell }\left( \gamma _{i}^{k}\left( t\right) \right) =d\exp \left(
te_{i}^{k}\right) _{e_{0}\wedge e_{1}\wedge e_{2}}\left( e_{j}^{\ell
}\right) =\left. \tfrac{d}{ds}\right\vert _{0}\exp \left( te_{i}^{k}\right)
\gamma _{j}^{\ell }\left( s\right) \text{.}\qedhere
\end{equation*}
\end{proof}

\medskip

From now on in this section we consider $k$ and $\ell $ $ \operatorname{mod}3$.

\begin{lemma}
\label{lema1}For $\ell =0,1,2$ and $j=3,\dots ,7$ we have that 
\begin{equation*}
\nabla _{e_{j}^{\ell }}\sigma _{3}=X\left( e_{j},e_{\ell +1},e_{\ell
+2}\right) +\delta _{j3}e_{\ell }\text{\ \ \ \ \ and\ \ \ \ \ }\nabla
_{e_{j}^{\ell }}\nabla _{E_{j}^{\ell }}\sigma _{3}=\left( \delta
_{j3}-1\right) e_{3}\text{.}
\end{equation*}%
Moreover, the first expression equals $0$ for $j=3$.
\end{lemma}

\begin{proof}
By the previous lemma with $i=j$, $k=\ell $, we write%
\begin{equation*}
P_{t,s}=\exp \left( te_{j}^{\ell }\right) \gamma _{j}^{\ell }\left( s\right)
=\gamma _{j}^{\ell }\left( t+s\right) \text{.}
\end{equation*}

We evaluate $\sigma _{3}$ at $P_{t,s}$ and, since $X\left( e_{\ell },e_{\ell
+1},e_{\ell +2}\right) =e_{3}$, we obtain%
\begin{align*}
\sigma _{3}\left( P_{t,s}\right) & =\sigma _{3}\left( \left( \cos \left(
t+s\right) \,e_{\ell }+\sin \left( t+s\right) \,e_{j}\right) \wedge e_{\ell
+1}\wedge e_{\ell +2}\right) \\
& =X\left( \cos \left( t+s\right) \,e_{\ell }+\sin \left( t+s\right)
\,e_{j},e_{\ell +1},e_{\ell +2}\right) \\
& =\cos \left( t+s\right) \,e_{3}+\sin \left( t+s\right) \,X\left(
e_{j},e_{\ell +1},e_{\ell +2}\right) \text{.}
\end{align*}%
We compute $S_{1}$, $S_{2}$ and $\pi _{t}$ as in Lemma \ref{derivadask,n},
obtaining%
\begin{align}
S_{1}& =\left. \tfrac{d}{ds}\right\vert _{0}\left( \cos s\,e_{3}+\sin
s\,X\left( e_{j},e_{\ell +1},e_{\ell +2}\right) \right) =X\left(
e_{j},e_{\ell +1},e_{\ell +2}\right) \text{,}  \label{DS1lema1} \\
S_{2}& =\left. \tfrac{\partial ^{2}}{\partial t\partial s}\right\vert
_{\left( 0,0\right) }\left( \cos \left( t+s\right) \,e_{3}+\sin \left(
t+s\right) \,X\left( e_{j},e_{\ell +1},e_{\ell +2}\right) \right) =-e_{3}%
\text{,}  \notag \\
\pi _{t}& = \operatorname{id}-\left( \cos t\,e_{\ell }+\sin t\,e_{j}\right) \otimes
\left( \cos t\,e^{\ell }+\sin t\,e^{j}\right)  \notag \\
& \ \ \ -\left( e_{\ell +1}\otimes e^{\ell +1}+e_{\ell +2}\otimes e^{\ell
+2}\right) \text{.}  \notag
\end{align}%
In particular, $\pi _{0}= \operatorname{id}-e_{0}\otimes e^{0}-e_{1}\otimes
e^{1}-e_{2}\otimes e^{2}$.

Then $\pi _{0}\circ S_{1}=X\left( e_{j},e_{\ell +1},e_{\ell +2}\right) -A$,
where 
\begin{align*}
A& =\left( e_{0}\otimes e^{0}+e_{1}\otimes e^{1}+e_{2}\otimes e^{2}\right)
X\left( e_{j},e_{\ell +1},e_{\ell +2}\right) \\
& =\left( e_{\ell }\otimes e^{\ell }\right) X\left( e_{j},e_{\ell
+1},e_{\ell +2}\right) =-\left( e_{\ell }\otimes e^{j}\right) X\left(
e_{\ell },e_{\ell +1},e_{\ell +2}\right) \\
& =-\left( e_{\ell }\otimes e^{j}\right) e_{3}=-\delta _{j3}e_{\ell }
\end{align*}%
(we used (\ref{propiedadX}) and (\ref{DobleTriple})). Therefore the first
assertion of the lemma is true by (\ref{DS1}) with $t=0$. We note that $%
\nabla _{e_{j}^{\ell }}\sigma _{3}=0$ if $j=3$, since in this case $X\left(
e_{3},e_{\ell +1},e_{\ell +2}\right) =-e_{\ell }$ for $\ell =0,1,2$.

Now we verify the second identity using (\ref{DS2}). We compute%
\begin{align*}
\pi _{0}\circ \pi _{0}^{\prime }& =\left(  \operatorname{id}-e_{0}\otimes
e^{0}-e_{1}\otimes e^{1}-e_{2}\otimes e^{2}\right) \left( -e_{j}\otimes
e^{\ell }-e_{\ell }\otimes e^{j}\right) \\
& =\left( -e_{j}\otimes e^{\ell }-e_{\ell }\otimes e^{j}\right) +\left(
e_{\ell }\otimes e^{j}\right) =-e_{j}\otimes e^{\ell }\text{.}
\end{align*}%
By property (\ref{propiedadX}),%
\begin{align*}
\pi _{0}\circ \pi _{0}^{\prime }\circ S_{1}& =-\left( e_{j}\otimes e^{\ell
}\right) X\left( e_{j},e_{\ell +1},e_{\ell +2}\right) =\left( e_{j}\otimes
e^{j}\right) X\left( e_{\ell },e_{\ell +1},e_{\ell +2}\right) \\
& =\left( e_{j}\otimes e^{j}\right) e_{3}=\delta _{j3}e_{j}\text{.}
\end{align*}%
Also,%
\begin{equation*}
\pi _{0}\circ S_{2}=\left(  \operatorname{id}-e_{0}\otimes e^{1}-e_{1}\otimes
e^{1}-e_{2}\otimes e^{2}\right) \left( -e_{3}\right) =-e_{3}\text{.}
\end{equation*}%
Adding the last two expressions yields that the second assertion of the
lemma is valid, since $\delta _{3j}e_{j}-e_{3}=0$ if $j=3$ and $\delta
_{3j}e_{j}-e_{3}=-e_{3}$ in the other cases.
\end{proof}

\begin{lemma}
\label{lema2} For $\ell =0,1,2$ and $i,j=3,\dots ,7$ with $i\neq j$ we have%
\begin{equation*}
\nabla _{e_{i}^{\ell }}\nabla _{E_{j}^{\ell }}\sigma _{3}=\delta _{j3}e_{i}%
\text{.}
\end{equation*}
\end{lemma}

\begin{proof}
Let $P_{t,s}$ be as in Lemma \ref{dercov2} with $k=\ell $ and $i\neq j$,
that is,%
\begin{equation*}
P_{t,s}=\exp \left( te_{i}^{\ell }\right) \gamma _{j}^{\ell }\left( s\right)
=\left( \cos t\,\cos s\,e_{\ell }+\sin t\,\cos s\,e_{i}+\sin s\,e_{j}\right)
\wedge e_{\ell +1}\wedge e_{\ell +2}\text{.}
\end{equation*}

We evaluate $\sigma _{3}$ at $P_{t,s}$ and, since $X\left( e_{\ell },e_{\ell
+1},e_{\ell +2}\right) =e_{3}$, we obtain%
\begin{align*}
\sigma _{3}\left( P_{t,s}\right) & =X\left( \cos t\,\cos s\,e_{\ell }+\sin
t\,\cos s\text{ }e_{i}+\sin s\,e_{j},e_{\ell +1},e_{\ell +2}\right) \\
& =\cos s\,\left( \cos t\,e_{3}+\sin t\,X\left( e_{i},e_{\ell +1},e_{\ell
+2}\right) \right) +\sin s\,X\left( e_{j},e_{\ell +1},e_{\ell +2}\right) 
\text{.}
\end{align*}%
In order to use Lemma \ref{derivadask,n}, we compute%
\begin{align*}
S_{1}& =\left. \tfrac{d}{ds}\right\vert _{0}\left( \cos s\,e_{3}+\sin
s\,X\left( e_{j},e_{\ell +1},e_{\ell +2}\right) \right) =X\left(
e_{j},e_{\ell +1},e_{\ell +2}\right) \text{,} \\
S_{2}& =\left. \tfrac{\partial ^{2}}{\partial t\partial s}\right\vert
_{\left( 0,0\right) }\sigma _{3}\left( P_{t,s}\right) =0\text{,} \\
\pi _{t}& = \operatorname{id}-\left( \cos t\,e_{\ell }+\sin t\,e_{i}\right) \otimes
\left( \cos t\,e^{\ell }+\sin t\,e^{i}\right) \\
& \ \ \ -\left( e_{\ell +1}\otimes e^{\ell +1}+e_{\ell +2}\otimes e^{\ell
+2}\right) \text{.}
\end{align*}%
Then, $\pi _{0}^{\prime }=-e_{i}\otimes e^{\ell }-e_{\ell }\otimes e^{i}$
and we have%
\begin{align*}
\pi _{0}\circ \pi _{0}^{\prime }& =\left(  \operatorname{id}-e_{0}\otimes
e^{1}-e_{1}\otimes e^{1}-e_{2}\otimes e^{2}\right) \left( -e_{i}\otimes
e^{\ell }-e_{\ell }\otimes e^{i}\right) \\
& =\left( -e_{i}\otimes e^{\ell }-e_{\ell }\otimes e^{i}\right) +\left(
e_{\ell }\otimes e^{i}\right) =-e_{i}\otimes e^{\ell }\text{.}
\end{align*}%
Thus, the identity of the statement is valid since%
\begin{align*}
\pi _{0}\circ \pi _{0}^{\prime }\circ S_{1}& =-\left( e_{i}\otimes e^{\ell
}\right) X\left( e_{j},e_{\ell +1},e_{\ell +2}\right) =\left( e_{i}\otimes
e^{j}\right) X\left( e_{\ell },e_{\ell +1},e_{\ell +2}\right) \\
& =\left( e_{i}\otimes e^{j}\right) e_{3}=\delta _{j3}e_{i}\text{.}\qedhere
\end{align*}
\end{proof}

\begin{lemma}
\label{lema3} For $k,\ell =0,1,2$, with $k\neq \ell $, and $i,j=3,\dots ,7$
we have%
\begin{equation*}
\nabla _{e_{i}^{k}}\nabla _{E_{j}^{\ell }}\sigma _{3}=r_{k,\ell }\left( 
 \operatorname{id}-e_{k}\otimes e^{k}-e_{\ell }\otimes e^{\ell }\right) X\left(
e_{i},e_{j},e_{m}\right) \text{,}
\end{equation*}%
where $m\in \left\{ 0,1,2\right\} $ and $r_{k,\ell }=\pm 1$ satisfy $%
e_{k}\wedge e_{\ell }\wedge e_{m}=r_{k,\ell }e_{0}\wedge e_{1}\wedge e_{2}$.
\end{lemma}

\begin{proof}
For $k\neq \ell $, by Lemma \ref{dercov2}, we consider the parametrized
surface given by%
\begin{equation*}
P_{t,s}=\exp \left( te_{i}^{k}\right) \gamma _{j}^{\ell }\left( s\right)
=r_{k,\ell }\left( \cos t\,e_{k}+\sin t\,e_{i}\right) \wedge \left( \cos
s\,e_{\ell }+\sin s\,e_{j}\right) \wedge e_{m}\text{,}
\end{equation*}%
with index $k$, $\ell $ $ \operatorname{mod}3$ as in the previous lemmas. We evaluate 
$\sigma _{3}$ at $P_{t,s}$:%
\begin{equation*}
\sigma _{3}\left( P_{t,s}\right) =r_{k,\ell }X\left( \cos t\,e_{k}+\sin
t\,e_{i},\cos s\,e_{\ell }+\sin s\,e_{j},e_{m}\right) \text{.}
\end{equation*}%
In order to apply Lemma \ref{derivadask,n}, we compute%
\begin{align*}
S_{1}& =\left. \tfrac{d}{ds}\right\vert _{0}\left( r_{k,\ell }\left( \cos
s\,X\left( e_{k},e_{\ell },e_{m}\right) +\sin s\,X\left(
e_{k},e_{j},e_{m}\right) \right) \right)  \\
& =r_{k,\ell }X\left( e_{k},e_{j},e_{m}\right) \text{,} \\
S_{2}& =\left. \tfrac{\partial ^{2}}{\partial t\partial s}\right\vert
_{\left( 0,0\right) }\sigma _{3}\left( P_{t,s}\right) =r_{k,\ell }X\left(
e_{i},e_{j},e_{m}\right) \text{,} \\
\pi _{t}& = \operatorname{id}-\left( \cos t\,e_{k}+\sin t\,e_{i}\right) \otimes
\left( \cos t\,e^{k}+\sin t\,e^{i}\right) -e_{\ell }\otimes e^{\ell
}-e_{m}\otimes e^{m}\text{.}
\end{align*}

Hence, $\pi _{0}\circ \pi _{0}^{\prime }=\pi _{0}\circ \left( -e_{i}\otimes
e^{k}-e_{k}\otimes e^{i}\right) =-e_{i}\otimes e^{k}$ and%
\begin{equation*}
\pi _{0}\circ \pi _{0}^{\prime }\circ S_{1}=-r_{k,\ell }\left( e_{i}\otimes
e^{k}\right) X\left( e_{k},e_{j},e_{m}\right) =0\text{.}
\end{equation*}%
Now,%
\begin{equation*}
\pi _{0}\circ S_{2}=r_{k,\ell }\left(  \operatorname{id}-e_{k}\otimes e^{k}-e_{\ell
}\otimes e^{\ell }\right) X\left( e_{i},e_{j},e_{m}\right) \text{,}
\end{equation*}%
since $e^{m}X\left( e_{i},e_{j},e_{m}\right) =0$. Then the identity follows
from Lemma \ref{derivadask,n}.
\end{proof}

\bigskip

The canonical action of $SO\left( 8\right) $ over $G\left( 3,8\right) $
induces an action on the total space of the sphere bundle $%
E_{3,8}^{1}\rightarrow G\left( 3,8\right) $ as in (\ref{bundle1}), given by $%
g\cdot \left( u\wedge v\wedge w,x\right) =\left( gu\wedge gv\wedge
gw,gx\right) $, for all $u\wedge v\wedge w\in G\left( 3,8\right) $ and $%
x\perp u\wedge v\wedge w$.

Let $ \operatorname{Spin}\left( 7\right) $ be the automorphism group of the triple
cross product $X_{3,8}$, which is known to be contained in $SO\left(
8\right) $. The following result (Theorem 8.2 in \cite{salamon}) will be
useful later.

\begin{theorem}
\cite{salamon}\label{Spin7} The group $ \operatorname{Spin}\left( 7\right) $ acts
transitively on the set%
\begin{equation*}
\left\{ \left( u,v,w,x\right) \in \mathbb{R}^{8}\mid u,v,w,X_{3,8}\left(
u,v,w\right) ,x\text{ are orthonormal}\right\} \text{.}
\end{equation*}
\end{theorem}

\begin{proposition}
\label{sarmonica} The section $\sigma _{3}$ of $E_{3,8}^{1}\rightarrow
G\left( 3,8\right) $ is vertically harmonic.
\end{proposition}

\begin{proof}
We use the criterion given in Proposition \ref{PacoCarmelo}~(a). We look for
a function $f:G\left( 3,8\right) \rightarrow \mathbb{R}$ such that $\Delta
\sigma _{3}=f\sigma _{3}$. We have that%
\begin{equation}
\left( \Delta \sigma _{3}\right) \left( e_{0}\wedge e_{1}\wedge e_{2}\right)
=\sum\nolimits_{\ell =0}^{2}\sum\nolimits_{j=3}^{7}\nabla _{e_{j}^{\ell
}}\nabla _{E_{j}^{\ell }}\sigma _{3}\text{.}  \label{Delta3}
\end{equation}

By Lemma \ref{lema1}, the term $\nabla _{e_{j}^{\ell }}\nabla _{E_{j}^{\ell
}}\sigma _{3}$ is non-zero only when $j\neq 3$, and in this case it is equal
to $-e_{3}$. So,%
\begin{equation*}
\left( \Delta \sigma _{3}\right) \left( e_{0}\wedge e_{1}\wedge e_{2}\right)
=\sum\nolimits_{\ell =0}^{2}\sum\nolimits_{j=4}^{7}\left( -e_{3}\right)
=-12e_{3}=-12\sigma _{3}\left( e_{0}\wedge e_{1}\wedge e_{2}\right) \text{.}
\end{equation*}

As an immediate consequence of Theorem \ref{Spin7}, the isometric action of
the group $ \operatorname{Spin}\left( 7\right) $ on $G\left( 3,8\right) $ is
transitive. Also, the section $\sigma _{3}$ and the connection on $E_{3,8}$
are invariant by the action. Therefore, $\Delta \sigma _{3}=-12\sigma _{3}$
on $G\left( 3,8\right) $.
\end{proof}

\begin{lemma}
\label{curvatura} For $k,\ell =0,1,2$ and $i,j=3,\dots ,7$, the curvature $%
R_{e_{i}^{k}e_{j}^{\ell }}\sigma _{3}$ equals $\delta _{i3}e_{j}-\delta
_{j3}e_{i}$ if $k=\ell $ and $0$ if $k\neq \ell $.
\end{lemma}

\begin{proof}
By Lemma \ref{lema2} we know that $\nabla _{e_{i}^{\ell }}\nabla
_{E_{j}^{\ell }}\sigma _{3}=\delta _{j3}e_{i}$ if $i\neq j$. Then, by (\ref%
{corIgualCero}),%
\begin{equation*}
R_{e_{i}^{\ell }e_{j}^{\ell }}\sigma _{3}=-\nabla _{e_{i}^{\ell }}\nabla
_{E_{j}^{\ell }}\sigma _{3}+\nabla _{e_{j}^{\ell }}\nabla _{E_{i}^{\ell
}}\sigma _{3}=-\delta _{j3}e_{i}+\delta _{i3}e_{j}\text{,}
\end{equation*}%
as stated (even if $i=j$). When $k\neq \ell $, we get by Lemma \ref{lema3}
that 
\begin{align*}
R_{e_{i}^{k}e_{j}^{\ell }}\sigma _{3}& =-\nabla _{e_{i}^{k}}\nabla
_{E_{j}^{\ell }}\sigma _{3}+\nabla _{e_{j}^{\ell }}\nabla _{E_{i}^{k}}\sigma
_{3} \\
& =\left( r_{k,\ell }+r_{\ell ,k}\right) \left(  \operatorname{id}-e_{k}\otimes
e^{k}-e_{\ell }\otimes e^{\ell }\right) X\left( e_{j},e_{i},e_{m}\right) =0%
\text{,}
\end{align*}%
since $r_{\ell ,k}X\left( e_{j},e_{i},e_{m}\right) =-r_{k,\ell }X\left(
e_{i},e_{j},e_{m}\right) $.
\end{proof}

\begin{theorem}
\label{TeoSigma3} The section $\sigma _{3}:G\left( 3,8\right) \rightarrow
E_{3,8}^{1}$ is a harmonic map.
\end{theorem}

\begin{proof}
Proposition \ref{sarmonica} asserts that $\sigma _{3}$ is vertically
harmonic; so, by Proposition \ref{PacoCarmelo}~(b), it remains to prove that
the $1$-form $\mathcal{R}_{\sigma _{3}}$ defined in (\ref{Rsigma}) is
identically zero. By invariance of the action of $ \operatorname{Spin}\left( 7\right) 
$, it is sufficient to verify this at $e_{0}\wedge e_{1}\wedge e_{2}$.

By the previous lemma and Lemma \ref{lema1}, for all $\ell ,i$ we have that%
\begin{align}
\mathcal{R}_{\sigma _{3}}\left( e_{i}^{\ell }\right) &
=\sum\nolimits_{k=0}^{2}\sum\nolimits_{j=3}^{7}\left\langle R_{e_{i}^{\ell
},e_{j}^{k}}\sigma _{3},\nabla _{e_{j}^{k}}\sigma _{3}\right\rangle 
\label{Rsigma3=0} \\
& =\sum\nolimits_{j=3}^{7}\left\langle \delta _{i3}e_{j}-\delta
_{j3}e_{i},X\left( e_{j},e_{\ell +1},e_{\ell +2}\right) +\delta _{j3}e_{\ell
}\right\rangle   \notag \\
& =\sum\nolimits_{j=4}^{7}\left\langle \delta _{i3}e_{j},X\left(
e_{j},e_{\ell +1},e_{\ell +2}\right) +\delta _{j3}e_{\ell }\right\rangle =0%
\text{.}  \notag
\end{align}%
(note that $X\left( e_{j},e_{\ell +1},e_{\ell +2}\right) +\delta
_{j3}e_{\ell }=0$ for $j=3$ since $X\left( e_{3},e_{\ell +1},e_{\ell
+2}\right) =-e_{\ell }$).
\end{proof}

\subsection{Harmonicity of $\protect\sigma _{2}$}

We recall from \cite{harvey} one of the equivalent definitions of the
exceptional Lie group $G_{2}$: 
\begin{equation*}
G_{2}=\left\{ g\in GL\left( 7,\mathbb{R}\right) \mid g\left( x\times
y\right) =g\left( x\right) \times g\left( y\right) \right\} \text{,}
\end{equation*}%
where $\times $ is the double cross product and $\mathbb{R}^{7}\equiv  \operatorname{%
Im}\mathbb{O}$. By Problem 9~(b) on page 121 of \cite{harvey}, the group
acts transitively on $G\left( 2,7\right) $, and by isometries, since $%
G_{2}\subset SO\left( 7\right) $. Therefore, its elements determine
morphisms of the Riemannian vector bundle $E_{2,7}\rightarrow G\left(
2,7\right) $. Clearly, the section $\sigma _{2}$ is preserved by this
action. The identity (\ref{DobleTriple}) provides a natural inclusion of $%
G_{2}$ in $ \operatorname{Spin}\left( 7\right) $ as a subgroup.

\bigskip

The proof that the section $\sigma _{2}$ is a harmonic map follows from the
harmonicity of the section $\sigma _{3}$. We consider the immersion%
\begin{equation}
\phi :G\left( 2,7\right) \rightarrow G\left( 3,8\right) \text{,}\ \ \ \ \phi
\left( u\wedge v\right) =e_{0}\wedge u\wedge v\text{,}  \label{phi1}
\end{equation}%
for all orthonormal $u,v$ in $ \operatorname{Im}\mathbb{O}$ and the morphism of
Riemannian vector bundles 
\begin{equation}
\Phi :E_{2,7}\rightarrow E_{3,8}\text{,}\ \ \ \ \Phi \left( u\wedge
v,x\right) =\left( e_{0}\wedge u\wedge v,x\right) \text{,}  \label{Phi}
\end{equation}%
where we identify $\mathbb{R}^{7}=e_{0}^{\bot }\subset \mathbb{R}^{8}$. Now,
since $X\left( e_{0},u,v\right) =u\times v$ for any $u$, $v$, the following
diagram is commutative:%
\begin{equation}
\begin{array}{ccc}
E_{2,7}^{1} & \overset{\Phi }{\longrightarrow } & E_{3,8}^{1} \\ 
\uparrow \sigma _{2} &  & \uparrow \sigma _{3} \\ 
G\left( 2,7\right) & \overset{\phi }{\longrightarrow } & G\left( 3,8\right) 
\text{.}%
\end{array}
\label{diagrama}
\end{equation}

\begin{theorem}
\label{TeoSigma2}The section $\sigma _{2}$ is a harmonic map.
\end{theorem}

\begin{proof}
Using that $\Phi $ is a Riemannian bundle morphism and that $\phi :G\left(
2,7\right) \rightarrow G\left( 3,8\right) $ is totally geodesic, one obtains%
\begin{equation*}
\left( \Delta _{E_{3,8}}\sigma _{3}\right) \circ \phi =-4\sigma _{3}\circ
\phi +\Phi \circ \left( \Delta _{E_{2,7}}\sigma _{2}\right) \text{.}
\end{equation*}%
Since we know from the proof of Proposition \ref{sarmonica} that $\Delta
\sigma _{3}=-12\sigma _{3}$, we deduce that $\Delta _{E_{2,7}}\sigma
_{2}=-8\sigma _{2}$, since $\Phi $ is one to one. Similarly, we can verify
that $\mathcal{R}_{\sigma _{2}}=\phi ^{\ast }\mathcal{R}_{\sigma _{3}}$.
Consequently $\sigma _{2}$ is a harmonic map by Proposition \ref{PacoCarmelo}%
.
\end{proof}


\subsection{Non-existence of parallel sections}

Here we prove Proposition \ref{ParallelSections}.

\begin{proof}[Proof of Proposition\textbf{\ }\protect\ref{ParallelSections}]

Let us assume that there is a parallel section $V$ of $E_{3,8}^{1}$. We put $%
P_{o}=e_{0}\wedge e_{1}\wedge e_{2}$ and $V\left( P_{o}\right) =\left(
P_{o},v\right) $. In particular, $v\perp e_{\ell }$ for $\ell =0,1,2$ and $v$
has unit norm. Let $w\in \mathbb{R}^{8}$ be a unit vector orthogonal to $v$
and to $P_{o}$.

For each $a\in \mathbb{R}$ we consider the orthonormal basis $\left\{
e_{0},u_{a},v_{a}\right\} $ of $P_{o}$, where%
\begin{equation*}
u_{a}=\cos a\text{ }e_{1}+\sin a\text{ }e_{2}\text{\ \ \ \ \ and\ \ \ \ \ }%
v_{a}=-\sin a\text{ }e_{1}+\cos a\text{ }e_{2}\text{,}
\end{equation*}%
and the curve $\gamma _{a}$ in $G\left( 3,8\right) $ defined by%
\begin{equation*}
\gamma _{a}\left( t\right) =e_{0}\wedge \left( \cos t\text{ }u_{a}+\sin t%
\text{ }v\right) \wedge \left( \cos t\text{ }v_{a}+\sin t\text{ }w\right) 
\text{.}
\end{equation*}

Let us see that $V\left( \gamma _{a}\left( t\right) \right) =-\sin t$ $%
u_{a}+\cos t$ $v$. Since $V\left( \gamma _{a}\left( 0\right) \right) =v$, it
suffices to show that the expression in the right hand side is parallel
along $\gamma _{a}$. We compute%
\begin{equation*}
\left. \tfrac{D}{ds}\right\vert _{t}\left( -\sin s\text{ }u_{a}+\cos s\text{ 
}v\right) =\pi _{t}\left( -\cos t\text{ }u_{a}-\sin t\text{ }v\right) \text{,%
}
\end{equation*}%
where $\pi _{t}$ is the orthogonal projection onto $\gamma _{a}\left(
t\right) ^{\bot }$. That is, if we call%
\begin{equation*}
z=-\cos t\text{ }u_{a}-\sin t\text{ }v\text{,\ \ \ }z_{1}=\cos t\text{ }%
u_{a}+\sin t\text{ }v\text{, \ \ \ }z_{2}=\cos t\text{ }v_{a}+\sin t\text{ }w%
\text{,}
\end{equation*}%
we have%
\begin{equation*}
\left. \tfrac{D}{ds}\right\vert _{t}\left( -\sin s\text{ }u_{a}+\cos s\text{ 
}v\right) =z-\left\langle z,e_{0}\right\rangle e_{0}-\left\langle
z,z_{1}\right\rangle z_{1}-\left\langle z,z_{2}\right\rangle z_{2}=0\text{.}
\end{equation*}

Now,\ since $V$ is a parallel section by the hypothesis, if we fix $0<t<\pi $%
, we have that $V\circ c$ is parallel along the curve $c:\mathbb{R}%
\rightarrow G\left( 3,8\right) $, $c\left( a\right) =\gamma _{a}\left(
t\right) $. But%
\begin{equation*}
\tfrac{D}{da}\left( V\left( c\left( a\right) \right) \right) =\pi _{a}\left( 
\tfrac{d}{da}V\left( c\left( a\right) \right) \right) =\pi _{a}\left( -\sin t%
\text{ }v_{a}\right) \neq 0\text{,}
\end{equation*}%
since $\frac{d}{da}u_{a}=v_{a}$ and $v_{a}$ is not in the kernel of $\pi
_{a} $. Thus, we arrived at a contradiction. Taking $t$ small enough, we can
see that there are no parallel sections on any neighborhood of $P_{o}$. The
assertion for $E_{2,7}^{1}$ follows using the diagram (\ref{diagrama}).
\end{proof}

\section[Normal complex section]{The energy of the normal complex section of 
$G\left( 2,8\right) $ associated with the triple cross product}

\subsection{Normal skew-symmetric sections}

For a plane $P\in G\left( 2,8\right) $, recall from (\ref{SkewP}) the
definition of $ \operatorname{Skew}_{P}\left( \mathbb{R}^{8}\right) $. On the vector
bundle $E\rightarrow G\left( 2,8\right) $ of (\ref{E2,8})\ we have a
canonical metric connection: Let $\Pi _{P}$ be the orthogonal projection of $%
 \operatorname{Skew}\left( \mathbb{R}^{8}\right) $ onto $ \operatorname{Skew}_{P}\left( 
\mathbb{R}^{8}\right) $. For a vector field $Y$ on $G\left( 2,8\right) $ and
a section $\mathfrak{T}\in \Gamma \left( G\left( 2,8\right) ,E\right) $ we
set%
\begin{equation*}
\left( \nabla _{Y}\mathfrak{T}\right) _{P}=\left( P,\Pi _{P}\left( \left(
dT\right) _{P}\left( Y_{P}\right) \right) \right) \text{,}
\end{equation*}%
where $\mathfrak{T}\left( P\right) =\left( P,T\left( P\right) \right) \in
E_{P}$ with $T:G\left( 2,8\right) \rightarrow  \operatorname{Skew}\left( \mathbb{R}%
^{8}\right) \subset \mathbb{R}^{8\times 8}$. It is easy to see that $\nabla $
is a metric connection on $E$, whose covariant derivative is the one given
in (\ref{DerCovJ}).

Next we present two lemmas that will be useful later.

\begin{lemma}
Let $P\in G\left( 2,8\right) $ and let $\pi _{P}$ be the orthogonal
projection of $\mathbb{R}^{8}$ onto $P^{\bot }$. Then $\Pi _{P}\left(
T\right) =\pi _{P}\circ T\circ \pi _{P}$ for all $T\in  \operatorname{Skew}\left( 
\mathbb{R}^{8}\right) $.
\end{lemma}

\begin{proof}
We call $\mathcal{P}\left( T\right) =\pi _{P}\circ T\circ \pi _{P}$, which
is in $ \operatorname{Skew}\left( \mathbb{R}^{8}\right) $ since $\pi _{P}$ is self
adjoint. First we show that $\mathcal{P}$ is an orthogonal projection, that
is, $\mathcal{P}^{2}=\mathcal{P}$ and that $\mathcal{P}$ is self-adjoint.
The former is true since $\left( \pi _{P}\right) ^{2}=\pi _{P}$ and for the
latter, given $S,T\in  \operatorname{Skew}\left( \mathbb{R}^{8}\right) $, we compute%
\begin{equation*}
\left\langle \mathcal{P}S,T\right\rangle =-\tfrac{1}{6}  \operatorname{tr}\left( \pi
_{P}\circ S\circ \pi _{P}\circ T\right) =-\tfrac{1}{6}  \operatorname{tr}\left(
S\circ \pi _{P}\circ T\circ \pi _{P}\right) =\left\langle S,\mathcal{P}%
T\right\rangle \text{.}
\end{equation*}

We verify now that the set of fixed points of $\mathcal{P}$ is $ \operatorname{Skew}%
_{P}\left( \mathbb{R}^{8}\right) $. If $\mathcal{P}\left( S\right) =S$, then
they coincide on $P$, in particular, they vanish on $P$ and so $S\in  \operatorname{%
Skew}_{P}\left( \mathbb{R}^{8}\right) $. Now, if $S\ $is in this subspace,
then $S$ and $\mathcal{P}\left( S\right) $ are $0$ on $P$ and so $S$
preserves $P^{\bot }$, since $S$ is skew symmetric. Hence, given $v\in
P^{\bot }$, $\mathcal{P}\left( S\right) v=\pi _{p}\left( Sv\right) =Sv$.
Thus, $\mathcal{P}\left( S\right) =S$. Consequently, $\mathcal{P}\left(
S\right) =S$ if and only if $S\in  \operatorname{Skew}_{P}\left( \mathbb{R}%
^{8}\right) $, as desired.
\end{proof}

\begin{lemma}
\label{derivadas2,8}Let $P:\mathbb{R}^{2}\rightarrow G\left( 2,8\right) $ be
a parametrized surface and let $\mathfrak{T}=\left(  \operatorname{id},T\right) $ be
a smooth section of $E\rightarrow G\left( 2,8\right) $. We denote%
\begin{equation}
T_{1}=\left. \tfrac{d}{ds}\right\vert _{0}T\left( P_{0,s}\right) \text{ }\ \
\ \ \ \text{and }\ \ \ \ \ T_{2}=\left. \tfrac{\partial ^{2}}{\partial
t\partial s}\right\vert _{\left( 0,0\right) }\left( P_{t,s}\right)
\label{T1T2}
\end{equation}%
\emph{(}both in $\mathbb{R}^{8\times 8}$\emph{)} and by $\pi _{t}$ the
orthogonal projection onto $\left( P_{t,0}\right) ^{\bot }$. Then%
\begin{equation}
\left. \tfrac{D}{ds}\right\vert _{0}\mathfrak{T}(P_{0,s})=\left( P_{0,0},\pi
_{0}\circ T_{1}\circ \pi _{0}\right) \text{,}  \label{DJ1}
\end{equation}%
\begin{equation}
\left. \tfrac{D^{2}}{dtds}\right\vert _{\left( 0,0\right) }\mathfrak{T}%
(P_{t,s})=\left( P_{0,0},\pi _{0}\circ \left( \pi _{0}^{\prime }\circ
T_{1}+T_{2}+T_{1}\circ \pi _{0}^{\prime }\right) \circ \pi _{0}\right) \text{%
.}  \label{DJ2}
\end{equation}
\end{lemma}

\begin{proof}
It is similar to the proof of Lemma \ref{derivadask,n}, using the previous
lemma.
\end{proof}

\subsection{Harmonicity of $\mathfrak{J}$}

In this section we call $X$ the triple cross product $X_{3,8}$. Also,
sometimes we shall omit to specify the foot planes of the elements of $E$.

Theorem \ref{Spin7} implies that the group $ \operatorname{Spin}\left( 7\right) $
acts transitively on $G\left( 2,8\right) $ through $g(u\wedge v)=gu\wedge gv$%
. It also acts on $E^{1}$ by means of $g\left( u\wedge v,T\right) =\left(
g\left( u\wedge v\right) ,gTg^{-1}\right) $, for $T\in  \operatorname{Skew}_{u\wedge
v}\left( \mathbb{R}^{8}\right) $, since $ \operatorname{Spin}\left( 7\right) \subset
SO\left( 8\right) $. It is easy to verify that the section $\mathfrak{J}%
:G\left( 2,8\right) \rightarrow E^{1}$ is invariant by this action.

The curves%
\begin{equation}
\gamma _{j}^{0}\left( t\right) =\left( \cos t\,e_{0}+\sin t\,e_{j}\right)
\wedge e_{1}\;\;\;\text{and}\;\;\;\gamma _{j}^{1}\left( t\right)
=e_{0}\wedge \left( \cos t\,e_{1}+\sin t\,e_{j}\right) \text{ }
\label{geodesica28}
\end{equation}%
or equivalently, $\gamma _{j}^{\ell }\left( t\right) =\left( -1\right)
^{\ell }\left( \cos t\,e_{\ell }+\sin t\,e_{j}\right) \wedge e_{\ell +1}$,
for $j=2,\dots ,7$, with $\ell =0,1$ $ \operatorname{mod}2$, are geodesics of $%
G\left( 2,8\right) $. In the following we will consider always $\ell $ $%
 \operatorname{mod}2$.

We consider vector fields $E_{j}^{\ell }$ on a normal neighborhood of $%
e_{0}\wedge e_{1}$ in $G\left( 2,8\right) $ analogous to the ones on $%
G\left( 3,8\right) $ in the previous section.

\begin{lemma}
\label{Derivadas28} Let $\mathfrak{T}$ be a section of $E^{1}\rightarrow
G\left( 2,8\right) $. For $k,\ell =0,1$ and $i,j=2,\dots ,7$, we have that%
\begin{equation*}
\nabla _{e_{i}^{k}}\nabla _{E_{j}^{\ell }}\mathfrak{T}=\left. \dfrac{%
D^{\gamma _{i}^{k}}}{dt}\right\vert _{0}\left( \left. \dfrac{D}{ds}%
\right\vert _{0}\mathfrak{T}(P_{t,s})\right) \text{,}
\end{equation*}%
where $P_{t,s}=\exp (te_{i}^{k})\gamma _{j}^{\ell }\left( s\right) $.
\end{lemma}

\begin{proof}
It is similar to the proof of Lemma \ref{dercov2}.
\end{proof}

In the next lemma we consider the section $\mathfrak{J}$ associated with the
triple cross product defined in (\ref{Jota}) and follow the notation of (\ref%
{T1T2}).

\begin{lemma}
Let $P:\mathbb{R}^{2}\rightarrow G\left( 2,8\right) $ be the parametrized
surface given by $P_{t,s}=\exp (te_{i}^{k})\gamma _{j}^{\ell }\left(
s\right) $. For the section $\mathfrak{J}=\left(  \operatorname{id},J\right) $ of $%
E\rightarrow G\left( 2,8\right) $, we have that%
\begin{equation}
J_{1}=\left. \tfrac{d}{ds}\right\vert _{0}J_{P_{0,s}}=\left( -1\right)
^{\ell }J_{e_{j}\wedge e_{\ell +1}}\text{.}  \label{J1}
\end{equation}
\end{lemma}

\begin{proof}
We evaluate $P_{0,s}=\gamma _{j}^{\ell }\left( s\right) $ at\ $w\in \mathbb{R%
}^{8}$ and obtain%
\begin{align}
J_{P_{0,s}}\left( w\right) & =\left( -1\right) ^{\ell }X\left( \cos
s\,e_{\ell }+\sin s\,e_{j},e_{\ell +1},w\right)  \label{JP} \\
& =\left( -1\right) ^{\ell }\left( \cos s\,X\left( e_{\ell },e_{\ell
+1},w\right) +\sin s\,X\left( e_{j},e_{\ell +1},w\right) \right)  \notag \\
& =\cos s\,J_{e_{0}\wedge e_{1}}\left( w\right) +\left( -1\right) ^{\ell
}\sin s\,J_{e_{j}\wedge e_{\ell +1}}\left( w\right) \text{,}  \notag
\end{align}%
from which we conclude that the desired expression is valid.
\end{proof}

\smallskip

Let $\left\{ e^{0},\dots ,e^{7}\right\} $ be the canonical dual basis of $%
\mathbb{R}^{8}$. We deduce from (\ref{propiedadX}) that $J_{e_{i}\wedge
e_{j}}=-J_{e_{j}\wedge e_{i}}$ and also the following properties of $J$,
which will be useful later:%
\begin{equation}
e^{k}\circ J_{e_{i}\wedge e_{j}}=-e^{i}\circ J_{e_{k}\wedge
e_{j}}=-e^{j}\circ J_{e_{i}\wedge e_{k}}=-\left\langle J_{e_{i}\wedge
e_{j}}\left( e_{k}\right) ,\cdot \right\rangle \text{.}  \label{properties}
\end{equation}

For convenience, given $A,B\in $ End~$\left( \mathbb{R}^{8}\right) $, we
denote $A\odot B=AB+BA$.

\begin{lemma}
\label{lema1J} For $\ell =0,1$ and $j=2,\dots ,7$, we have%
\begin{gather}
\nabla _{e_{j}^{\ell }}\mathfrak{J}=\left( -1\right) ^{\ell }\left(
J_{e_{j}\wedge e_{\ell +1}}-\left( e_{\ell }\otimes e^{\ell }\right) \odot
J_{e_{j}\wedge e_{\ell +1}}\right) \text{,}  \label{derivada1} \\
\nabla _{e_{j}^{\ell }}\nabla _{E_{j}^{\ell }}\mathfrak{J}=-J_{e_{0}\wedge
e_{1}}+\left( e_{j}\otimes e^{j}\right) \odot J_{e_{0}\wedge e_{1}}\text{.}
\label{derivada2}
\end{gather}
\end{lemma}

\begin{proof}
By Lemma \ref{Derivadas28}, we consider the parametrized surface%
\begin{equation*}
P_{t,s}=\exp \left( te_{j}^{\ell }\right) \gamma _{j}^{\ell }\left( s\right)
=\gamma _{j}^{\ell }\left( t+s\right) \text{.}
\end{equation*}%
By (\ref{JP}) we have that that $\mathfrak{J}\left( P_{t,s}\right) =\left(
P_{t,s},J_{P_{t,s}}\right) $ equals%
\begin{equation*}
\left( P_{t,s},\cos \left( t+s\right) J_{e_{0}\wedge e_{1}}+\left( -1\right)
^{\ell }\sin \left( t+s\right) J_{e_{j}\wedge e_{\ell +1}}\right) \text{.}
\end{equation*}%
In order to obtain the first expression, we use (\ref{DJ1}) with $T_{1}=J_{1}
$ given by (\ref{J1}) and 
\begin{equation*}
\pi _{t}= \operatorname{id}-\left( \cos t\,e_{\ell }+\sin t\,e_{j}\right) \otimes
\left( \cos t\,e^{\ell }+\sin t\,e^{j}\right) -e_{\ell +1}\otimes e^{\ell +1}%
\text{.}
\end{equation*}%
We compute%
\begin{equation}
J_{1}\pi _{0}=\left( -1\right) ^{\ell }J_{e_{j}\wedge e_{\ell +1}}\left( 
 \operatorname{id}-e_{\ell }\otimes e^{\ell }\right) \text{,}  \label{J1pi0}
\end{equation}%
since $J_{e_{j}\wedge e_{\ell +1}}\left( e_{\ell +1}\right) =0$. Then 
\begin{align*}
\pi _{0}J_{1}\pi _{0}& =\left(  \operatorname{id}-e_{0}\otimes e^{0}-e_{1}\otimes
e^{1}\right) \left( -1\right) ^{\ell }J_{e_{j}\wedge e_{\ell +1}}\left( 
 \operatorname{id}-e_{\ell }\otimes e^{\ell }\right)  \\
& =\left( -1\right) ^{\ell }\left(  \operatorname{id}-e_{\ell }\otimes e^{\ell
}\right) \left( J_{e_{j}\wedge e_{\ell +1}}-J_{e_{j}\wedge e_{\ell
+1}}\left( e_{\ell }\otimes e^{\ell }\right) \right) \text{,}
\end{align*}%
which equals the desired expression in (\ref{derivada1}), since $e^{\ell
}J_{e_{j}\wedge e_{\ell +1}}\left( e_{\ell }\right) =0$.

Now we verify the second identity. We use (\ref{DJ2}) with $J_{1}$ as above
and%
\begin{equation*}
J_{2}=\left. \tfrac{\partial ^{2}}{\partial t\partial s}\right\vert _{\left(
0,0\right) }J_{P_{t,s}}=-J_{e_{0}\wedge e_{1}}\text{.}
\end{equation*}

We have that $\pi _{0}^{\prime }=-e_{j}\otimes e^{\ell }-e_{\ell }\otimes
e^{j}$ and 
\begin{align*}
\pi _{0}\pi _{0}^{\prime }& =\left(  \operatorname{id}-e_{0}\otimes
e^{0}-e_{1}\otimes e^{1}\right) \left( -e_{j}\otimes e^{\ell }-e_{\ell
}\otimes e^{j}\right)  \\
& =-e_{j}\otimes e^{\ell }-e_{\ell }\otimes e^{j}+e_{\ell }\otimes
e^{j}=-e_{j}\otimes e^{\ell }\text{.}
\end{align*}%
Using (\ref{J1pi0}) and properties (\ref{properties}) we have%
\begin{align*}
\pi _{0}\pi _{0}^{\prime }J_{1}\pi _{0}& =-\left( e_{j}\otimes e^{\ell
}\right) \left( -1\right) ^{\ell }J_{e_{j}\wedge e_{\ell +1}}\left(  \operatorname{id}%
-e_{\ell }\otimes e^{\ell }\right)  \\
& =\left( -1\right) ^{\ell }\left( e_{j}\otimes e^{j}\right) J_{e_{\ell
}\wedge e_{\ell +1}}\left(  \operatorname{id}-e_{\ell }\otimes e^{\ell }\right)
=\left( e_{j}\otimes e^{j}\right) \left( J_{e_{0}\wedge e_{1}}\right) \text{,%
}
\end{align*}%
\begin{align*}
\pi _{0}J_{2}\pi _{0}& =\left(  \operatorname{id}-e_{0}\otimes e^{0}-e_{1}\otimes
e^{1}\right) \left( -J_{e_{0}\wedge e_{1}}\right) \left(  \operatorname{id}%
-e_{0}\otimes e^{0}-e_{1}\otimes e^{1}\right)  \\
& =-J_{e_{0}\wedge e_{1}}\text{.}
\end{align*}

Now, 
\begin{equation*}
\pi _{0}J_{1}\pi _{0}^{\prime }\pi _{0}=-\left( \pi _{0}\pi _{0}^{\prime
}J_{1}\pi _{0}\right) ^{t}=-\left( \left( e_{j}\otimes e^{j}\right) \left(
J_{e_{0}\wedge e_{1}}\right) \right) ^{t}=\left( J_{e_{0}\wedge
e_{1}}\right) \left( e_{j}\otimes e^{j}\right) \text{,}
\end{equation*}%
since $\pi _{0}$ and $\pi _{0}^{\prime }$ are self adjoint and $J_{1}$ is
skew symmetric. Therefore%
\begin{equation*}
\pi _{0}\left( \pi _{0}^{\prime }J_{1}+J_{2}+J_{1}\pi _{0}^{\prime }\right)
\pi _{0}=\left( e_{j}\otimes e^{j}\right) J_{e_{0}\wedge
e_{1}}-J_{e_{0}\wedge e_{1}}+J_{e_{0}\wedge e_{1}}\left( e_{j}\otimes
e^{j}\right) \text{, }
\end{equation*}%
as we wanted to see in (\ref{derivada2}).
\end{proof}

\begin{lemma}
\label{lema2J} For $\ell =0,1$, $i,j=2,\dots ,7$ with $i\neq j$, we have%
\begin{equation*}
\nabla _{e_{i}^{\ell }}\nabla _{E_{j}^{\ell }}\mathfrak{J}=\left(
e_{i}\otimes e^{j}\right) J_{e_{0}\wedge e_{1}}+J_{e_{0}\wedge e_{1}}\left(
e_{j}\otimes e^{i}\right) \text{.}
\end{equation*}
\end{lemma}

\begin{proof}
If $i\neq j$, by Lemma \ref{Derivadas28}, we consider the parametrized
surface%
\begin{equation*}
P_{t,s}=\exp \left( te_{i}^{\ell }\right) \gamma _{j}^{\ell }\left( s\right)
=\left( -1\right) ^{\ell }\left( \cos t\cos s\,e_{\ell }+\sin t\cos
s\,e_{i}+\sin s\,e_{j}\right) \wedge e_{\ell +1}\text{.}
\end{equation*}%
Then, $\mathfrak{J}\left( P_{t,s}\right) =\left( P_{t,s},J_{P_{t,s}}\right) $%
, with%
\begin{equation*}
J_{P_{t,s}}=\left( -1\right) ^{\ell }\left( \cos t\,\cos s\,J_{e_{\ell
}\wedge e_{\ell +1}}+\sin t\cos s\,J_{e_{i}\wedge e_{\ell +1}}+\sin
s\,J_{e_{j}\wedge e_{\ell +1}}\right) \text{.}
\end{equation*}

By Lemma \ref{derivadas2,8},%
\begin{equation*}
\nabla _{e_{i}^{\ell }}\nabla _{E_{j}^{\ell }}\mathfrak{J=}\left. \tfrac{%
D^{2}}{dtds}\right\vert _{\left( 0,0\right) }\mathfrak{J}=\left( P_{0,0},\pi
_{0}\pi _{0}^{\prime }J_{1}\pi _{0}+\pi _{0}J_{2}\pi _{0}+\pi _{0}J_{1}\pi
_{0}^{\prime }\pi _{0}\right) 
\end{equation*}%
where $J_{1}$ is given in (\ref{J1}), $J_{2}=\left. \frac{\partial ^{2}}{%
\partial t\partial s}\right\vert _{\left( 0,0\right) }J_{P_{t,s}}=0$ and%
\begin{equation*}
\pi _{t}= \operatorname{id}-\left( \cos t\,e_{\ell }+\sin t\,e_{i}\right) \otimes
\left( \cos t\,e^{\ell }+\sin t\,e^{i}\right) -e_{\ell +1}\otimes e^{\ell +1}%
\text{.}
\end{equation*}%
Then, $\pi _{0}^{\prime }=-e_{i}\otimes e^{\ell }-e_{\ell }\otimes e^{i}$.
We compute%
\begin{align*}
\pi _{0}\pi _{0}^{\prime }& =\left(  \operatorname{id}-e_{0}\otimes
e^{0}-e_{1}\otimes e^{1}\right) \left( -e_{i}\otimes e^{\ell }-e_{\ell
}\otimes e^{i}\right)  \\
& =-e_{i}\otimes e^{\ell }-e_{\ell }\otimes e^{i}+e_{\ell }\otimes
e^{i}=-e_{i}\otimes e^{\ell }\text{.}
\end{align*}%
Using\ (\ref{J1pi0}) and properties (\ref{properties}) we have%
\begin{align*}
\pi _{0}\pi _{0}^{\prime }J_{1}\pi _{0}& =-\left( e_{i}\otimes e^{\ell
}\right) \left( -1\right) ^{\ell }J_{e_{j}\wedge e_{\ell +1}}\left(  \operatorname{id}%
-e_{\ell }\otimes e^{\ell }\right)  \\
& =\left( -1\right) ^{\ell }\left( e_{i}\otimes e^{j}\right) J_{e_{\ell
}\wedge e_{\ell +1}}=\left( e_{i}\otimes e^{j}\right) J_{e_{0}\wedge e_{1}}%
\text{.}
\end{align*}%
Now, as in the previous lemma, 
\begin{equation*}
\pi _{0}J_{1}\pi _{0}^{\prime }\pi _{0}=-\left( \pi _{0}\pi _{0}^{\prime
}J_{1}\pi _{0}\right) ^{t}=-\left( \left( e_{i}\otimes e^{j}\right)
J_{e_{0}\wedge e_{1}}\right) ^{t}=\left( J_{e_{0}\wedge e_{1}}\right) \left(
e_{j}\otimes e^{i}\right) \text{.}
\end{equation*}%
Therefore, as desired, 
\begin{equation*}
\pi _{0}\left( \pi _{0}^{\prime }J_{1}+J_{2}+J_{1}\pi _{0}^{\prime }\right)
\pi _{0}=\left( e_{i}\otimes e^{j}\right) J_{e_{0}\wedge
e_{1}}+J_{e_{0}\wedge e_{1}}\left( e_{j}\otimes e^{i}\right) \text{.}\qedhere
\end{equation*}
\end{proof}

\begin{lemma}
\label{lema3J}For $\ell =0,1$ and $i,j=2,\dots ,7$, we have%
\begin{equation*}
\nabla _{e_{i}^{\ell +1}}\nabla _{E_{j}^{\ell }}\mathfrak{J}=\left(
-1\right) ^{\ell +1}\left( \pi _{0}\circ J_{e_{i}\wedge e_{j}}\circ \pi
_{0}\right) \text{,}
\end{equation*}%
where $\pi _{0}$ is as above the orthogonal projection onto $\left(
e_{0}\wedge e_{1}\right) ^{\bot }$.
\end{lemma}

\begin{proof}
If $i\neq j$, by Lemma \ref{Derivadas28}, we consider the parametrized
surface%
\begin{equation*}
P_{t,s}=\exp \left( te_{i}^{\ell +1}\right) \gamma _{j}^{\ell }\left(
s\right) =\left( -1\right) ^{\ell +1}\left( \cos t\,e_{\ell +1}+\sin
t\,e_{i}\right) \wedge \left( \cos s\,e_{\ell }+\sin s\,e_{j}\right) \text{.}
\end{equation*}

Then, $\left( -1\right) ^{\ell +1}\mathfrak{J}\left( P_{t,s}\right) $ is
equal to%
\begin{equation*}
\cos s\left( \cos t\,J_{e_{\ell +1}\wedge e_{\ell }}+\sin t\,J_{e_{i}\wedge
e_{\ell }}\right) +\sin s\left( \cos t\,J_{e_{\ell +1}\wedge e_{j}}+\sin
t\,J_{e_{i}\wedge e_{j}}\right) \text{.}
\end{equation*}

By Lemma \ref{derivadas2,8},%
\begin{equation*}
\nabla _{e_{i}^{\ell +1}}\nabla _{E_{j}^{\ell }}\mathfrak{J}=\left. \tfrac{%
D^{2}}{dtds}\right\vert _{\left( 0,0\right) }\mathfrak{J}=\left( P_{0,0},\pi
_{0}\pi _{0}^{\prime }J_{1}\pi _{0}+\pi _{0}J_{2}\pi _{0}+\pi _{0}J_{1}\pi
_{0}^{\prime }\pi _{0}\right) \text{,}
\end{equation*}%
where $J_{1}$ is given by (\ref{J1}),%
\begin{equation*}
J_{2}=\left. \tfrac{\partial ^{2}}{\partial t\partial s}\right\vert _{\left(
0,0\right) }J_{P_{t,s}}=\left( -1\right) ^{\ell +1}J_{e_{i}\wedge e_{j}}
\end{equation*}%
and%
\begin{equation*}
\pi _{t}= \operatorname{id}-\left( \cos t\,e_{\ell +1}+\sin t\,e_{i}\right) \otimes
\left( \cos t\,e^{\ell +1}+\sin t\,e^{i}\right) -e_{\ell }\otimes e^{\ell }%
\text{.}
\end{equation*}%
Then, $\pi _{0}^{\prime }=-e_{i}\otimes e^{\ell +1}-e_{\ell +1}\otimes e^{i}$%
. We compute%
\begin{align*}
\pi _{0}\pi _{0}^{\prime }& =\left(  \operatorname{id}-e_{0}\otimes
e^{0}-e_{1}\otimes e^{1}\right) \left( -e_{i}\otimes e^{\ell +1}-e_{\ell
+1}\otimes e^{i}\right)  \\
& =-e_{i}\otimes e^{\ell +1}-e_{\ell +1}\otimes e^{i}+e_{\ell +1}\otimes
e^{i}=-e_{i}\otimes e^{\ell +1}\text{.}
\end{align*}%
Also, using\ (\ref{J1pi0}), we obtain%
\begin{equation*}
\pi _{0}\pi _{0}^{\prime }J_{1}\pi _{0}=\left( -1\right) ^{\ell +1}\left(
e_{i}\otimes e^{\ell +1}\right) J_{e_{j}\wedge e_{\ell +1}}\left(  \operatorname{id}%
-e_{\ell }\otimes e^{\ell }\right) =0\text{,}
\end{equation*}%
Now, $\pi _{0}J_{1}\pi _{0}^{\prime }\pi _{0}=-\left( \pi _{0}\pi
_{0}^{\prime }J_{1}\pi _{0}\right) ^{t}=0$. Consequently, 
\begin{equation*}
\pi _{0}\left( \pi _{0}^{\prime }J_{1}+J_{2}+J_{1}\pi _{0}^{\prime }\right)
\pi _{0}=\pi _{0}J_{2}\pi _{0}=\left( -1\right) ^{\ell +1}\pi
_{0}J_{e_{i}\wedge e_{j}}\pi _{0}\text{,}
\end{equation*}%
which implies the statement of the lemma.
\end{proof}

\begin{proposition}
\label{jotava}The section $\mathfrak{J}:G(2,8)\rightarrow E_{2,8}^{1}$ is
vertically harmonic.
\end{proposition}

\begin{proof}
We will see that $\Delta \mathfrak{J}(u\wedge v)=f\,\mathfrak{J}(u\wedge v)$
for some smooth real function $f$. So, $\mathfrak{J}$ will be vertically
harmonic by Proposition \ref{PacoCarmelo}~(a). We first check it at $u\wedge
v=e_{0}\wedge e_{1}$. We consider the geodesic $\gamma _{j}^{\ell }$ as in (%
\ref{geodesica28}). We compute%
\begin{equation*}
\Delta \mathfrak{J}(e_{0}\wedge e_{1})=\sum\nolimits_{\ell
=0}^{1}\sum\nolimits_{j=2}^{7}\nabla _{e_{j}^{\ell }}\nabla _{E_{j}^{\ell }}%
\mathfrak{J}\text{.}
\end{equation*}%
By Lemma \ref{lema1J}, for $\ell =0,1$ and $j=2,\dots ,7$ we have%
\begin{equation*}
\nabla _{e_{j}^{\ell }}\nabla _{E_{j}^{\ell }}\mathfrak{J}=-J_{e_{0}\wedge
e_{1}}+\left( e_{j}\otimes e^{j}\right) J_{e_{0}\wedge e_{1}}+J_{e_{0}\wedge
e_{1}}\left( e_{j}\otimes e^{j}\right) ,
\end{equation*}%
which does not depend on $\ell $. Consequently,%
\begin{align*}
\Delta \mathfrak{J}(e_{0}\wedge e_{1})& =2\sum\nolimits_{j=2}^{7}\left(
-J_{e_{0}\wedge e_{1}}+\left( e_{j}\otimes e^{j}\right) J_{e_{0}\wedge
e_{1}}+J_{e_{0}\wedge e_{1}}\left( e_{j}\otimes e^{j}\right) \right)  \\
& =2\left( -6\,J_{e_{0}\wedge e_{1}}+J_{e_{0}\wedge e_{1}}+J_{e_{0}\wedge
e_{1}}\right) =-8\,J_{e_{0}\wedge e_{1}}\text{,}
\end{align*}%
since $\sum_{j=2}^{7}\left( e_{j}\otimes e^{j}\right) =0$ on $ \operatorname{span}%
\left\{ e_{0},e_{1}\right\} $ and the identity on its orthogonal complement.

Finally, one can show, using Theorem \ref{Spin7}, that $ \operatorname{Spin}\left(
7\right) $ acts transitively on $G\left( 2,8\right) $ and it preserves $%
\mathfrak{J}$, and this implies that $\Delta \mathfrak{J}=-8\,\mathfrak{J}$.
\end{proof}

\smallskip

In the following lemma we compute the curvatures that we shall use later in
the proof of Theorem \ref{Jarmo}. For brevity, we denote $%
A^{i,j}=e_{i}\otimes e^{j}-e_{j}\otimes e^{i}$.

\begin{lemma}
\label{curvaturaJ}For $\ell =0,1$ and $2\leq i,j\leq 7$ we have that $\left(
R_{e_{i}^{0}e_{j}^{1}}\mathfrak{J}\right) _{e_{0}\wedge e_{1}}=0$ and%
\begin{equation*}
\left( R_{e_{i}^{\ell }e_{j}^{\ell }}\mathfrak{J}\right) _{e_{0}\wedge
e_{1}}=\left[ J_{e_{0}\wedge e_{1}},A^{i,j}\right] \text{.}
\end{equation*}
\end{lemma}

\begin{proof}
By Lemma \ref{lema3J} and the analogue of (\ref{corIgualCero}), 
\begin{equation*}
R_{e_{i}^{0}e_{j}^{1}}\mathfrak{J}=-\nabla _{e_{i}^{0}}\nabla _{E_{j}^{1}}%
\mathfrak{J}+\nabla _{e_{i}^{1}}\nabla _{E_{j}^{0}}\mathfrak{J}=-\pi
_{0}\circ \left( J_{e_{i}\wedge e_{j}}+J_{e_{j}\wedge e_{i}}\right) \circ
\pi _{0}=0\text{.}
\end{equation*}

Also, by Lemma \ref{lema2J} we have that 
\begin{align*}
R_{e_{i}^{\ell }e_{j}^{\ell }}\mathfrak{J}& =-\nabla _{e_{i}^{\ell }}\nabla
_{E_{j}^{\ell }}\mathfrak{J}+\nabla _{e_{j}^{\ell }}\nabla _{E_{i}^{\ell }}%
\mathfrak{J} \\
& =-\left( e_{i}\otimes e^{j}\right) J_{e_{0}\wedge e_{1}}-J_{e_{0}\wedge
e_{1}}\left( e_{j}\otimes e^{i}\right)  \\
& \ \ \ \ +\left( e_{j}\otimes e^{i}\right) J_{e_{0}\wedge
e_{1}}+J_{e_{0}\wedge e_{1}}\left( e_{i}\otimes e^{j}\right)  \\
& =\left( e_{j}\otimes e^{i}-e_{i}\otimes e^{j}\right) J_{e_{0}\wedge
e_{1}}-J_{e_{0}\wedge e_{1}}\left( e_{j}\otimes e^{i}-e_{i}\otimes
e^{j}\right) \text{,}
\end{align*}%
which coincides with the stated formula.
\end{proof}

\begin{proof}[Proof of Theorem \protect\ref{Jarmo}]
We know from Proposition \ref{jotava} that $\mathfrak{J}$ is vertically
harmonic. Then, by Proposition \ref{PacoCarmelo} it suffices to verify that $%
\mathcal{R}_{\mathfrak{J}}$ vanishes. By the invariance by the action of $%
 \operatorname{Spin}\left( 7\right) $, we need to check it only at $e_{0}\wedge e_{1}$%
. Since $R_{e_{i}^{\ell }e_{j}^{\ell +1}}\mathfrak{J}=0$ by Lemma \ref%
{curvaturaJ}, then 
\begin{equation*}
\mathcal{R}_{\mathfrak{J}}\left( e_{i}^{\ell }\right)
=\sum\nolimits_{j=2}^{7}\left\langle R_{e_{i}^{\ell },e_{j}^{\ell }}%
\mathfrak{J},\nabla _{e_{j}^{\ell }}\mathfrak{J}\right\rangle \text{.}
\end{equation*}

We will show than each summand vanishes.

By (\ref{derivada1}) and Lemma \ref{curvaturaJ}, since $ \operatorname{ad}{}_{Z}$ is
skew-symmetric for all $Z$ for the inner product on $ \operatorname{Skew}{}%
_{e_{0}\wedge e_{1}}$, we obtain%
\begin{eqnarray*}
\left( -1\right) ^{\ell }\left\langle R_{e_{i}^{\ell },e_{j}^{\ell }}%
\mathfrak{J},\nabla _{e_{j}^{\ell }}\mathfrak{J}\right\rangle
&=&\left\langle \left[ J_{e_{0}\wedge e_{1}},A^{i,j}\right] ,J_{e_{j}\wedge
e_{\ell +1}}-\left( e_{\ell }\otimes e^{\ell }\right) \odot J_{e_{j}\wedge
e_{\ell +1}}\right\rangle \\
&=&S_{1}+S_{2}\text{,}
\end{eqnarray*}%
where%
\begin{eqnarray*}
S_{1} &=&-\left\langle A^{i,j},\left[ J_{e_{0}\wedge e_{1}},J_{e_{j}\wedge
e_{\ell +1}}\right] \right\rangle \text{,} \\
S_{2} &=&\left\langle A^{i,j},\left[ J_{e_{0}\wedge e_{1}},\left( e_{\ell
}\otimes e^{\ell }\right) J_{e_{j}\wedge e_{\ell +1}}+J_{e_{j}\wedge e_{\ell
+1}}\left( e_{\ell }\otimes e^{\ell }\right) \right] \right\rangle \text{.}
\end{eqnarray*}

We observe that if $B$ is skew-symmetric, then 
\begin{equation}
\left\langle A^{i,j},B\right\rangle =2\left\langle Be_{j},e_{i}\right\rangle
=-2e^{i}\left( Be_{j}\right) =2e^{j}\left( Be_{i}\right) \text{.}
\label{short}
\end{equation}

We deduce from (\ref{properties}) and Corollary 6.22 of \cite{salamon} that%
\begin{eqnarray*}
J_{e_{j}\wedge e_{\ell +1}}J_{e_{0}\wedge e_{1}}e_{j} &=&J_{e_{j}\wedge
e_{\ell +1}}J_{e_{j}\wedge e_{0}}e_{1}=X\left( e_{j},e_{\ell +1},X\left(
e_{j},e_{0},e_{1}\right) \right) \\
&=&-\left\langle e_{\ell +1},e_{0}\right\rangle e_{1}+\left\langle e_{\ell
+1},e_{1}\right\rangle e_{0}=\left( -1\right) ^{\ell }e_{\ell }\text{.}
\end{eqnarray*}%
Consequently, 
\begin{eqnarray*}
S_{1} &=&2e^{i}\left( \left[ J_{e_{0}\wedge e_{1}},J_{e_{j}\wedge e_{\ell
+1}}\right] e_{j}\right) =2e^{i}\left( J_{e_{0}\wedge e_{1}}J_{e_{j}\wedge
e_{\ell +1}}e_{j}-J_{e_{j}\wedge e_{\ell +1}}J_{e_{0}\wedge
e_{1}}e_{j}\right) \\
&=&-2e^{i}\left( J_{e_{j}\wedge e_{\ell +1}}J_{e_{0}\wedge
e_{1}}e_{j}\right) =-2e^{i}\left( \left( -1\right) ^{\ell }e_{\ell }\right)
=0
\end{eqnarray*}%
and, using that $J_{e_{0}\wedge e_{1}}\left( e_{\ell }\right) $, $e^{\ell
}\left( e_{j}\right) $ and $e^{\ell }\left( J_{e_{0}\wedge e_{1}}\right) $
vanish, 
\begin{eqnarray*}
S_{2} &=&-2e^{i}\left( \left[ J_{e_{0}\wedge e_{1}},\left( e_{\ell }\otimes
e^{\ell }\right) J_{e_{j}\wedge e_{\ell +1}}+J_{e_{j}\wedge e_{\ell
+1}}\left( e_{\ell }\otimes e^{\ell }\right) \right] e_{j}\right) \\
&=&-2e^{i}\left( J_{e_{0}\wedge e_{1}}\left( \left( e_{\ell }\otimes e^{\ell
}\right) J_{e_{j}\wedge e_{\ell +1}}+J_{e_{j}\wedge e_{\ell +1}}\left(
e_{\ell }\otimes e^{\ell }\right) \right) e_{j}\right) \\
&&+2e^{i}\left( \left( \left( e_{\ell }\otimes e^{\ell }\right)
J_{e_{j}\wedge e_{\ell +1}}+J_{e_{j}\wedge e_{\ell +1}}\left( e_{\ell
}\otimes e^{\ell }\right) \right) J_{e_{0}\wedge e_{1}}e_{j}\right) \\
&=&2e^{i}\left( \left( e_{\ell }\otimes e^{\ell }\right) J_{e_{j}\wedge
e_{\ell +1}}J_{e_{0}\wedge e_{1}}e_{j}\right) \\
&=&2e^{i}\left( \left( e_{\ell }\otimes e^{\ell }\right) \left( -1\right)
^{\ell }e_{\ell }\right) =\left( -1\right) ^{\ell }2e^{i}\left( e_{\ell
}\right) =0\text{.}
\end{eqnarray*}

Thus, $\mathcal{R}_{\mathfrak{J}}(e_{i}^{\ell })$ vanishes for all $\ell
=0,1 $ and $i=2,\dots ,7$, as desired.
\end{proof}

\subsection{The energy of orthogonal almost complex structures of $S^{6}
\label{final}$}

We have defined the distinguished section $\mathfrak{J}$ in terms of the
triple cross product. Let us see what are the analogues if we consider the
other cross products (described in Subsection \ref{crossproduct}). We
observe that the case $\left( 1,2m\right) $ is empty and the case $\left(
m,m+1\right) $ is trivial (since in dimension $2$ there are exactly two
complex orthogonal structures, the one associated to the cross product and
the opposite).

Only the case $\left( r,n\right) =\left( 1,7\right) $ remains: the
assignment $u\in G\left( 1,7\right) \equiv S^{6}\mapsto J_{u}$, where $J_{u}$
is the orthogonal complex transformation $J_{u}$ in $u^{\bot }=T_{u}S^{6}$
given by $J_{u}\left( v\right) =u\times v$. This is the canonical almost
complex structure of $S^{6}$ and its energy was studied in \cite{CP}:
Theorem 4.1 (see also Proposition 4.7) asserts that this section is
harmonic, identifying $ \operatorname{Skew}_{u}\left( \mathbb{R}^{7}\right) $ with $%
\Lambda ^{2}\left( u^{\bot }\right) $. Then one concludes from Theorem \ref%
{Jarmo} that all the normal orthogonal complex sections of the Grassmannian
associated with cross products are harmonic maps as sections of the
corresponding fiber bundle, except for the empty case $\left( 1,2m\right) $.

On the other hand, in \cite{BLS} it is proved that $J$ has minimum energy
among all the almost complex orthogonal structures of $S^{6}$, and in
particular it is vertically harmonic among the orthogonal almost complex
structures (we comment that it contradicts \cite{WoodCrelle}, which states
that this structure is not even a local minimum). Note that this result is
strong, since it provides an absolute minimum, but it is not comparable with 
\cite{CP}. Indeed, in the latter only critical points are studied, but
regarding variations through arbitrary maps (not only through sections, i.e.
orthogonal almost complex structures). Also, the typical fiber in the case
studied in \cite{CP} is strictly larger than that of \cite{BLS} (consisting
of constant trace skew-symmetric and orthogonal skew-symmetric
transformations, respectively).

\bigskip

\noindent \textsc{f}a\textsc{maf} (Universidad
Nacional de C\'{o}rdoba) \ - \ \textsc{ciem} (Conicet)

\medskip

\noindent Ciudad Universitaria, (5000) C\'{o}rdoba, Argentina

\medskip

\noindent franciscoferraris032@hotmail.com, 

\noindent paomoas@unc.edu.ar,

\noindent salvai@famaf.unc.edu.ar

\end{document}